\documentclass[11pt]{article}

\usepackage[a4paper,margin=1in]{geometry}
\usepackage{amsmath,amsthm,amssymb,mathtools}
\usepackage{microtype}
\usepackage{hyperref}
\usepackage{authblk}
\usepackage{pgfplots}
\usepackage{float}
\usepackage{xurl}
\pgfplotsset{compat=1.18}
\usepackage[nameinlink,capitalize]{cleveref}

\newtheorem{theorem}{Theorem}[section]
\newtheorem{question}{Question}
\newtheorem{lemma}[theorem]{Lemma}
\newtheorem{proposition}[theorem]{Proposition}

\theoremstyle{definition}

\newtheorem{remark}[theorem]{Remark}

\title{Almost-Orthogonality in \texorpdfstring{$L^p$}{Lp} Spaces:
A Case Study with Grok}

\author[4]{
  Ziang Chen$^\ast$
}
\author[1]{
  Jaume de Dios Pont$^\dagger$
}
\author[2]{
  Paata Ivanisvili$^\dagger$
}
\author[3]{
  Jos\'e Madrid$^\dagger$
}
\author[4]{
  Haozhu Wang$^\ast$
}

\affil[1]{ETH Z\"urich, \texttt{jaume.dediospont@math.ethz.ch}}
\affil[2]{University of California, Irvine, \texttt{pivanisv@uci.edu}}
\affil[3]{Virginia Tech, \texttt{josemadrid@vt.edu}}
\affil[4]{xAI, \texttt{\{zchen,hwang\}@x.ai}}

\date{}

\begin{document}

\maketitle

\begin{center}
\small
Authors are listed alphabetically.\\
$^\dagger$Mathematical contributor,
$^\ast$Engineering contributor.
\end{center}
\begin{abstract}
Carbery proposed the following sharpened form of triangle inequality for many functions:  for any $p\ge 2$ and any finite sequence $(f_j)_j\subset L^p$ we have 
\[
\Big\|\sum_j f_j\Big\|_p \ \le\  \left(\sup_{j} \sum_{k} \alpha_{jk}^{\,c}\right)^{1/p'} \Big(\sum_j \|f_j\|_p^p\Big)^{1/p},
\]
where $c=2$, $1/p+1/p'=1$, and $\alpha_{jk}=\sqrt{\frac{\|f_{j}f_{k}\|_{p/2}}{\|f_{j}\|_{p}\|f_{k}\|_{p}}}$. 
In the first part of this paper we construct a counterexample showing that this inequality fails for every $p>2$. We then prove that if an estimate of the above form holds, the exponent must satisfy $c\le p'$. Finally, at the critical exponent $c=p'$, we establish the inequality for all integer values $p\ge 2$.

In the second part of the paper we obtain a sharp three-function bound
\[
\Big\|\sum_{j=1}^{3} f_j\Big\|_p \ \le\  \left(1+2\Gamma^{c(p)}\right)^{1/p'} \Big(\sum_{j=1}^{3} \|f_j\|_p^p\Big)^{1/p},
\]
where $p \geq 3$, $c(p) = \frac{2\ln(2)}{(p-2)\ln(3)+2\ln(2)}$ and   $\Gamma=\Gamma(f_1,f_2,f_3)\in[0,1]$ quantifies the degree of orthogonality among $f_1,f_2,f_3$. The exponent $c(p)$ is optimal, and improves upon the power $r(p) = \frac{6}{5p-4}$ obtained previously by Carlen, Frank, and Lieb. Some intermediate lemmas and inequalities appearing in this work were explored with the assistance of the large language model Grok.
\end{abstract}

\section{Introduction and the main results}

Let $(X,\mu)$ be a measure space. For $p\geq 1$ let $L^{p}(X, d\mu)$ be the space of $p$-integrable functions on $X$. 
In his study of almost orthogonality in Schatten-von Neumann  classes \cite{Carbery2009} Carbery was interested in the following question: 
\begin{question}\label{q1}
Under what conditions on the sequence of functions $(f_{j})_{j} \subset L^{p}$ do we have $\sum_{j} f_{j} \in L^{p}$\,? 
\end{question}

Notice that by the triangle inequality if $\sum_{j} \|f_{j}\|_{p}<\infty$ then necessarily $\sum_{j} f_{j} \in L^{p}$ without any further assumptions on the sequence $(f_{j})_{j}$. On the other hand if $f_{j}$'s have disjoint support then $\|\sum_{j=1}^{n} f_{j}\|_{p} = \left(\sum_{j=1}^{n} \|f_{j}\|_{p}^{p}\right)^{1/p}$ for any $n\geq 1$,  and, therefore, the finiteness of $\sum_{j} \|f_{j}\|_{p}^{p}$ implies $\sum_{j} f_{j} \in L^{p}$. Since $\left(\sum_{j} \|f_{j}\|_{p}^{p}\right)^{1/p}\leq \sum_{j} \|f_{j}\|_{p}$ a natural baseline assumption is $\sum_{j}\|f_{j}\|_{p}^{p} <\infty$. 

In what follows let $p\geq 2$. In \cite{Carbery2009} Carbery introduced nonnegative numbers 
\begin{align*}
    \alpha_{jk} := \sqrt{\frac{\|f_{j}f_{k}\|_{p/2}}{\|f_{j}\|_{p}\|f_{k}\|_{p}}} \in [0,1],
\end{align*}
where we set $\alpha_{kj}=0$ if  $\|f_{j}\|_{p} \|f_{k}\|_{p}=0$. The numbers $\alpha_{jk}$ measure the ``orthogonality'' between functions $f_{j}$ and $f_{k}$. Clearly if the $f_{j}, f_{k}$ have disjoint support then one can take $\alpha_{jk}=0$. 
In the $L^{2}$ case the Question~\ref{q1} has a simple answer in terms of $\alpha_{jk}$. Indeed, for any finite sequence $(f_{j})_{j} \subset L^{2}$ we have
\begin{align*}
    \Big\| \sum_{j} f_{j} \Big\|_{2}^{2} \leq  \sum_{j,k} \int_{X} |f_{j} f_{k}| d\mu \leq \sum_{j,k}  \alpha^{2}_{jk} \|f_{j}\|_{2} \|f_{k}\|_{2} \leq \|A\| \sum_{j} \|f_{j}\|_{2}^{2} \leq \Big(\sup_{k}\sum_{j}\alpha_{jk}^{2}\Big) \sum_{j}\|f_{j}\|_{2}^{2},
\end{align*}
where $A$ is the symmetric  matrix with entries $\alpha_{jk}^{2}$, and $\|A\|$ denotes its spectral radius. 
One ``naive'' way of interpolation between $L^{2}$ case and Cotlar--Stein lemma suggests that the following inequality may be true:

\begin{question}
Is it true that for $p\geq 2$ the inequality 
\begin{align}\label{qcarbery}
\Big\|\sum_j f_j\Big\|_p \ \le\  \left(\sup_{j} \sum_{k} \alpha_{jk}^{\,2}\right)^{1/p'} \Big(\sum_j \|f_j\|_p^p\Big)^{1/p},
\end{align}
where $1/p+1/p'=1$,
holds for all finite sequences $(f_{j})_{j}\subset L^{p}$? 
\end{question}

In \cite{Carbery2009} Carbery obtained a weaker form of \eqref{qcarbery} where $p\geq 2$ is any integer and $\alpha_{jk}^{2}$ are replaced by $\alpha_{jk}$. Since $\alpha_{jk}\in [0,1]$ therefore, having a larger power on $\alpha_{jk}$ only makes the inequality stronger. However, the advantage of Carbery's argument is that for even integer values of $p$ the inequality \eqref{qcarbery}, in fact, holds in Schatten-von Neumann $C_{p}$  classes with $\alpha_{jk}$ in place of $\alpha_{jk}^{2}$. 

Recently~\cite{CFIL2021} Carlen--Frank--Ivanisvili--Lieb verified the inequality \eqref{qcarbery} for two functions. In fact they proved a stronger estimate: for any $f,g \in L^{p}$ one has 
\begin{align}\label{CFIL}
    \|f+g\|_{p} \leq \left(1+\Gamma_{p}^{2/p}\right)^{1/p'}\left(\|f\|_{p}^{p}+\|g\|_{p}^{p}\right)^{1/p}, 
\end{align}
where $\Gamma_{p}=\Gamma_{p}(f,g) = \frac{2\|fg\|_{p/2}^{p/2}}{\|f\|_{p}^{p}+\|g\|_{p}^{p}}.$
Clearly $\Gamma^{2/p}_{p} \leq \alpha_{12}^{2}$, therefore, \eqref{CFIL} implies \eqref{qcarbery} for two functions. 

Later the estimate \eqref{CFIL} was further refined by Ivanisvili--Mooney in \cite{IvanisviliMooney2020} by showing that one has the bound 
\begin{align*}
   \|f+g\|_{p} \leq \frac{1}{2^{1/p}}\Big((1+\sqrt{1+\Gamma_{p}^{2}})^{1/p}+(1+\sqrt{1-\Gamma_{p}^{2}})^{1/p}\Big)(\|f\|_{p}^{p}+\|g\|_{p}^{p})^{1/p},   
\end{align*}
and the right-hand side of this inequality is the best possible one can have in terms of $\|f\|_{p}, \|g\|_{p}$, and $\|fg\|_{p/2}$.

In a subsequent work \cite{CFL2020} Carlen--Frank--Lieb showed that for any $n\geq 2$ and any $p\geq 2$ we have 
\begin{align}\label{CFL1}
\Big\|\sum_{j=1}^{n} f_j\Big\|_p \ \le\  \Big(1+(n-1)\Gamma_{p}(f_{1}, \ldots, f_{n})^{r}\Big)^{1/p'} \Big(\sum_{j=1}^{n} \|f_j\|_p^p\Big)^{1/p}
\end{align}
where 
\begin{align*}
    \Gamma_{p}(f_{1}, \ldots, f_{n}) = \frac{\Big \| \binom{n}{2}^{-1} \sum_{1\leq j<k\leq n} |f_{j}f_{k}|\Big \|_{p/2}^{p/2}}{n^{-1}\sum_{j=1}^{n}\|f_{j}\|_{p}^{p}}
\end{align*}
is $n$-function extension of the $\Gamma_{p}$ introduced in \eqref{CFIL}, and $r=r(n,p)=\frac{2n}{2n+(p-2)(2n-1)}$. Notice that $\Gamma_{p} \in [0,1]$. The inequalities \eqref{CFL1} and \eqref{qcarbery} do not imply each other. The goal of \eqref{CFL1} was to find a possible $n$-function analog of \eqref{CFIL} though we should point out that 2-function version of \eqref{CFL1} is actually weaker than \eqref{CFIL} because $r(2,p)=\frac{4}{4+3(p-2)}<\frac{2}{p}$. In \cite{CFL2020} Carlen--Frank--Lieb asked about the optimal (largest) power $r=r(n,p)$ one can have in \eqref{CFL1}. They provided an example showing that if \eqref{CFL1} holds for some $r>0$ then necessarily
\begin{align*}
    r \leq \frac{\ln(n-1) - \ln(n-2)}{(\frac{p}{2}-1) \ln(n) + \ln(n-1) - \frac{p}{2}\ln(n-2)} := \beta(n,p).
\end{align*}
The value $\beta(2,p)$ should be understood in the ``limiting'' sense:  $\lim_{n \to 2}\beta(n,p) = \frac{2}{p}$.

We provide a counterexample showing that for each $p>2$  the inequality \eqref{qcarbery} fails. In fact, the example shows that if the inequality holds with $\alpha_{jk}^{c}$ instead of $\alpha_{jk}^{2}$ in the right hand side of \eqref{qcarbery}  for some $c>0$ then necessarily $c \leq p'$. Our first theorem verifies the inequality at the critical power $c=p'$ for integers $p\geq 2$. 
\begin{theorem}\label{jose}
For any integer $p\geq 2$ the inequality 
\begin{align}\label{carbery0}
\Big\|\sum_j f_j\Big\|_p \ \le\  \left(\sup_{j} \sum_{k} \alpha_{jk}^{\,p'}\right)^{1/p'} \Big(\sum_j \|f_j\|_p^p\Big)^{1/p}
\end{align}
holds for all finite sequences $(f_{j})_{j}\subset L^{p}$. 
\end{theorem}
    
In the second part of this paper we prove a more general theorem stating that Carlen--Frank--Lieb's inequality \eqref{CFL1} holds with power $r = \beta(n,p)$ if and only if it holds for indicator functions $f_{j}=a_{j} \mathbf{1}_{A}$ for some common set $A$ of finite positive measure, and all $a_{j}\geq 0$, $j=1, \ldots, n$. 
\begin{theorem}\label{absnon}
Let $p\geq 2$ and $n>2$.  The inequality \eqref{CFL1} holds with $r=\beta(n,p)$ if and only if 
\begin{align}\label{num1}
    \Big(\sum_{j=1}^{n}a_{j} \Big)^{p'}\leq 1+(n-1)\Bigg[ n\Big(\binom{n}{2}^{-1}\sum_{1\leq i<j\leq n} a_{i}a_{j}\Big)^{p/2}\Bigg]^{\beta(n,p)} 
\end{align}
holds for all $a_{1}, \ldots, a_{n}\geq 0$ satisfying  $\sum_{j=1}^{n}a_{j}^{p}=1$
\end{theorem}
The inequality \eqref{num1} numerically seems correct. The difficulty of verifying it rigorously lies in the fact that the expression $\sum_{j=1}^{n}a_{j}$ achieves several local maxima in the interior of the compact set $\{\sum a_{j}^{p}=1\}\cap \{\sum_{i<j}a_{i}a_{j}=b\}$, see Lemma~\ref{teknika1} in Section~\ref{lieb1}. Moreover, the numerics show that for different values $p, n$, and $b,$ sometimes the global maximum is on the boundary and sometimes it is in the interior of that set.  

\begin{remark}
In Lemma \ref{teknika1} we prove that we can reduce inequality \eqref{num1} further, to the case where $a_j\in\{x,y,0\}$, for some $x,y\geq0$.   
\end{remark}

In the special case $n=3$ there is a certain {\em symmetry} which helps to decrease the number of parameters and local maxima points involved in verification of \eqref{num1}. In particular, we manage to sharpen Carlen--Frank--Lieb's inequality for three functions up to optimal power $\beta(3,p)$. 
\begin{theorem}\label{mth2}
    For $n=3$ and $p\geq 3$  the inequality \eqref{CFL1} holds with the optimal power $r=\beta(3,p)$.
\end{theorem}

\begin{remark}
There is nothing intrinsic about the assumption $p \ge 3$. In principle, the arguments presented here could likely be extended to cover the range $p \in [2,3)$ as well. However, while the proof initially proceeds in essentially the same way for $p \in [2,3)$ as for $p \ge 3$, the two cases begin to slightly diverge at a certain stage of the argument. In the interval $p \in [2,3)$ one must analyze a larger number of subcases, which would  lengthen the already technical parts of the paper. For this reason, we restrict the presentation to the case $p \ge 3$.
\end{remark}

\subsection{Structure of the paper and the use of the AI model Grok}

This paper developed in part through interactions with the AI model Grok. In particular, an early turning point was Grok Heavy's construction of a counterexample to Carbery's question. The version originally produced by Grok was substantially longer than the one presented here and contained several numerical inaccuracies; see the accompanying chat transcript in Appendix~\ref{grokchat}. Nevertheless, its underlying idea was correct. In the present paper, we give a more polished and streamlined version of this counterexample.

We also emphasize that, prior to Grok's construction, we were already aware that a counterexample should exist. However, the example we had previously found arose from a random brute-force search and did not provide conceptual insight, nor did it indicate the correct optimal exponent. By contrast, Grok's construction revealed a clear structural pattern, which in turn led to the optimal power $p'$.

In Section~\ref{jos01}, we prove the optimality of the exponent $p'$. For this part, the main ingredient of the argument, Lemma \ref{lem: dax1}, is essentially due to Grok; our contribution here consisted in reducing the proof of Theorem \ref{jose} to this lemma (refining Carbery's ideas), and carefully checking the details and editing the presentation.

The proofs of Theorems~\ref{absnon} and \ref{mth2} in Section~\ref{lieb1}, as well as the subsequent arguments in Section~\ref{mivedi}, were developed through a genuine human--AI collaboration. In these sections, we prove the Carlen--Frank--Lieb conjectured inequality for three functions for all $p \geq 3$, which appears to be the most technically difficult part of the paper. Grok 4.20 Heavy was used to generate a number of plots that were very helpful in guiding our intuition; many of these plots are included in the paper. Grok 4.20  was also used to verify numerous local inequalities arising in the argument; in total, we estimate that roughly thirty polynomial inequalities were checked in this way. To improve readability, we have moved these technical computations to Appendices~A and~B, and in Section~\ref{mivedi} we state only the corresponding lemmas. This allows the reader to first see the overall logical structure of the argument without interruption, and then, if desired, verify each technical computation separately in the appendices.

Without AI assistance, we likely would have chosen not to pursue such a lengthy and technical proof, since the computational verification alone would have required a substantial amount of time and effort. With Grok, however, this computational component became inexpensive and routine. At the same time, we wish to stress that the main conceptual idea of the proof, namely, the analysis based on counting roots of certain complicated functions, originated on the human side. The role of Grok was to help break this idea into a collection of simpler verifiable inequalities and to efficiently confirm those reductions.



\section{Proofs}

\subsection{Counterexamples}\label{jos1}

\begin{proposition}\label{jos0101}
Let $p\geq 2$. If the inequality 
\begin{align}\label{fineq}
\Big\|\sum_j f_j\Big\|_p \ \le\  \left(\sup_{j} \sum_{k} \alpha_{jk}^{\,c}\right)^{1/p'} \Big(\sum_j \|f_j\|_p^p\Big)^{1/p}
\end{align}
holds for some $c>0$ and all finite sequences $(f_{j})_{j} \in L^{p}$ then $c\leq p'$. 
\end{proposition}
\begin{proof}

    Consider a measure space $(X, d\mu)$  with $n \geq 3$ pairwise disjoint ``private'' sets $P_j$ each of measure $1$, and one ``common'' set $C$ of measure $1$. Define functions $f_j$ ($j = 1, \dots, n$) by $f_j = 1$ on $P_j$, $f_j = 1$ on $C$, and $f_j = 0$ elsewhere. Observe that $\|f_j\|_p^p = 2$.  Let $S := \sum_{j=1}^{n} f_j$. Then $S = 1$ on each $P_j$ and $S = n$ on $C$, so
\[
\|S\|_p^p = n + n^p.
\]
Next, notice that $\alpha_{jj}=1$. For each $j\neq k$, $\| f_{j}f_{k}\|_{p/2}^{p/2}=1$, hence $\alpha_{jk} = \frac{1}{2^{1/p}}$. Thus 
\begin{align*}
   \sup_{j}\sum_{k=1}^{n}\alpha_{jk}^{c} =  1 +  \frac{n-1}{2^{c/p}}.
\end{align*}
Thus the inequality \eqref{fineq} takes the form 
\begin{align*}
(n+n^{p})^{1/p}\leq \left(1 + \frac{n-1}{2^{c/p}}\right)^{1/p'}(2n)^{1/p}.
\end{align*}
This is the same as $n+n^{p} \leq \left(1+\frac{n-1}{2^{c/p}}\right)^{p-1} 2n$ which further can be rewritten as $\left(\frac{1+n^{p-1}}{2}\right)^{1/(p-1)}\leq 1+\frac{n-1}{2^{c/p}}$. Thus we obtain 
\begin{align}\label{tek1}
    2^{c/p} \leq \frac{n-1}{\left(\frac{1+n^{p-1}}{2}\right)^{1/(p-1)}-1}.
\end{align}
Taking $n \to \infty$ the inequality \eqref{tek1} yields $2^{c/p}\leq 2^{1/(p-1)}$, i.e., $c \leq \frac{p}{p-1}$. This finishes the proof of Proposition~\ref{jos0101}.
\end{proof}

\begin{remark}
The counterexample presented above is essentially the same as the one proposed to us by Grok, with two small modifications: we allow both  $n$ and $p$ to be arbitrary parameters rather than fixed values, and we polish the presentation. Treating $n$ as a free parameter and passing to the limit $n \to \infty$ enables us to recover the optimal exponent $p/(p-1)$.
\end{remark}

\subsection{The proof of Theorem~\ref{jose}}\label{jos01}
Without loss of generality assume $f_{j}\geq 0$ for all $j\geq 1$. We start with the following lemma. 

\begin{lemma}\label{lem: dax1}
Let $(\Omega, dx)$ be any measure space and let $K(x,y)$ be a nonnegative
symmetric integral kernel defined on $\Omega\times\Omega$.
Suppose $\kappa$ is defined by
\[
\kappa = \sup_{x}\int_{\Omega} K(x,y)^{\frac{p}{p-1}}\,dy.
\]
For any integer $p>1$ we have
\[
\int_{\Omega^{p}}
\prod_{1\leq i< j\leq p}K(x_{i},x_{j})^{\frac{2}{p-1}}
\prod_{s=1}^{p} F_{s}(x_{s})\,dx_{1}\cdots dx_{p}
\;\le\;
\kappa^{\,p-1}\,
\prod_{s=1}^{p}\|F_{s}\|_{p},
\]
for all $F_{s} \in L^{p}$, $s=1, \ldots, p$.
\end{lemma}
\begin{remark}
The proof below was proposed by Grok; see Appendix~\ref{grokchat} for a link to the conversation. After verifying its correctness, we present it essentially without modification.
\end{remark}
\begin{proof}
Since $K$ is symmetric and nonnegative, we rewrite the product in the integrand as
$$\prod_{1 \leq i < j \leq p} K(x_i, x_j)^{\frac{2}{p-1}} = \prod_{k=1}^p \Big( \prod_{j \neq k} K(x_k, x_j) \Big)^{\frac{1}{p-1}}.$$
Thus, the left-hand side is
$$\int_{\Omega^p} \prod_{k=1}^p \Big[ F_k(x_k) \Big( \prod_{j \neq k} K(x_k, x_j) \Big)^{\frac{1}{p-1}} \Big] \, dx_1 \cdots dx_p.$$
Define functions $h_k: \Omega^p \to \mathbb{R}$ for $k=1, \dots, p$ by
$$
h_k(x_1, \dots, x_p) = F_k(x_k) \Big( \prod_{j \neq k} K(x_k, x_j) \Big)^{\frac{1}{p-1}}.
$$
The integrand is $\prod_{k=1}^p h_k(\mathbf{x})$, so the integral is $\int_{\Omega^p} \prod_{k=1}^p h_k(\mathbf{x}) \, d\mathbf{x}$.
By the generalized Hölder's inequality on the space $L^p(\Omega^p)$ (with each exponent equal to $p$, since $\sum_{k=1}^p \frac{1}{p} = 1$),
$$
\Big| \int_{\Omega^p} \prod_{k=1}^p h_k(\mathbf{x}) \, d\mathbf{x} \Big| \leq \prod_{k=1}^p \| h_k \|_{L^p(\Omega^p)}.
$$
Without loss of generality, compute $\| h_1 \|_p$ (the others follow by symmetry, relabeling the $F_k$).
$$
\| h_1 \|_p^p = \int_{\Omega^p} \Big[ F_1(x_1) \Big( \prod_{j=2}^p K(x_1, x_j) \Big)^{\frac{1}{p-1}} \Big]^p \, dx_1 \cdots dx_p
= \int_{\Omega} F_1(x_1)^p \cdot \prod_{j=2}^p \int_{\Omega} K(x_1, x_j)^{\frac{p}{p-1}} \, dx_j\, dx_1,
$$
since the integrals separate. Each inner integral is $\int_\Omega K(x_1, y)^{\frac{p}{p-1}} \, dy \leq \kappa$, so
$$
\prod_{j=2}^p \int K(x_1, x_j)^{\frac{p}{p-1}} \, dx_j = \Big( \int_\Omega K(x_1, y)^{\frac{p}{p-1}} \, dy \Big)^{p-1} \leq \kappa^{p-1}.
$$
Thus,
$$\| h_1 \|_p^p \leq \kappa^{p-1} \int_\Omega F_1(x_1)^p \, dx_1 = \kappa^{p-1} \| F_1 \|_p^p,$$
so $\| h_1 \|_p \leq \kappa^{\frac{p-1}{p}} \| F_1 \|_p$.
Similarly for each $k$, $\| h_k \|_p \leq \kappa^{\frac{p-1}{p}} \| F_k \|_p$.
Therefore,
$$\int_{\Omega^p} \prod_{k=1}^p h_k(\mathbf{x}) \, d\mathbf{x} \leq \prod_{k=1}^p \left( \kappa^{\frac{p-1}{p}} \| F_k \|_p \right) = \kappa^{p-1} \prod_{k=1}^p \| F_k \|_p.$$
This completes the proof of Lemma~\ref{lem: dax1}.
\end{proof}

\vskip0.5cm 

Now we are ready to prove Theorem~\ref{jose}. 
\begin{remark}
The proof below was obtained with the assistance of Grok (see Appendix~\ref{grokchat}), after providing several specific directions to try. We present it here in a slightly polished form.
\end{remark}

Assume without loss of generality that the functions $f_j$ are nonnegative. Then,
$$
\Big\| \sum_j f_j \Big\|_p^p = \int \Big( \sum_j f_j(x) \Big)^p \, dx = \sum_{j_1, \dots, j_p} \int \prod_{s=1}^p f_{j_s}(x) \, dx.
$$
For a fixed multi-index $(j_1, \dots, j_p)$, set $g_s = f_{j_s}$ for $s = 1, \dots, p$. The integral becomes $\int \prod_{s=1}^p g_s(x) \, dx$.
Rewrite this as
$$\int \prod_{1 \leq i < j \leq p} (g_i(x) g_j(x))^{1/(p-1)} \, dx,$$
since each $g_s$ appears in exactly $p-1$ pairs, contributing an exponent of $(p-1) \cdot (1/(p-1)) = 1$.
Apply H\"older's inequality with equal exponents:
$$\int \prod_{1 \leq i < j \leq p} h_{ij}(x) \, dx \leq \prod_{1 \leq i < j \leq p} \| h_{ij} \|_{\binom{p}{2}},$$
where $\sum_{1 \leq i < j \leq p} \frac{1}{\binom{p}{2}} = 1$.
Set $h_{ij} = (g_i g_j)^{1/(p-1)}$. Then,
$$\| h_{ij} \|_{\binom{p}{2}} = \| g_i g_j \|_{\binom{p}{2} / (p-1)}^{1/(p-1)} = \| g_i g_j \|_{p/2}^{1/(p-1)},$$
since $\binom{p}{2} / (p-1) = p/2$.
By the given condition,
$$\| g_i g_j \|_{p/2} \leq \alpha_{j_i j_j}^2 \| g_i \|_p \| g_j \|_p,$$
so
$$\prod_{i < j} \| g_i g_j \|_{p/2}^{1/(p-1)} \leq \prod_{i < j} (\alpha_{j_i j_j}^2 \| g_i \|_p \| g_j \|_p)^{1/(p-1)} = \prod_{i < j} \alpha_{j_i j_j}^{2/(p-1)} \cdot \prod_{s=1}^p \| g_s \|_p,$$
where the last equality follows because each $\| g_s \|_p$ appears in $p-1$ terms, yielding exponent $(p-1) \cdot (1/(p-1)) = 1$.
Thus,
$$\int \prod_{s=1}^p f_{j_s}(x) \, dx \leq \prod_{1 \leq i < j \leq p} \alpha_{j_i j_j}^{2/(p-1)} \prod_{s=1}^p \| f_{j_s} \|_p.$$
Summing over all multi-indices gives
\begin{align}\label{fin1}
\Big\| \sum_j f_j \Big\|_p^p \leq \sum_{j_1, \dots, j_p} \Big( \prod_{1 \leq i < j \leq p} \alpha_{j_i j_j}^{2/(p-1)} \Big) \prod_{s=1}^p \| f_{j_s} \|_p.
\end{align}

The right-hand side of \eqref{fin1} is the discrete analog of the left-hand side of Lemma~\ref{lem: dax1} on the counting measure space over the indices $j$, with kernel $K(m, n) = \alpha_{m n}$ and functions $F_s(j) = \| f_j \|_p$ for each $s = 1, \dots, p$.
Here, $\kappa = \sup_j \sum_k \alpha_{j k}^{p/(p-1)} = \sup_j \sum_k \alpha_{j k}^{p'}$, and
\begin{align*}
\| F_s \|_p = \Big( \sum_j \| f_j \|_p^p \Big)^{1/p}, \quad \text{so} \quad \prod_{s=1}^p \| F_s \|_p = \sum_j \| f_j \|_p^p.  \end{align*}
By Lemma~\ref{lem: dax1},
$$
\sum_{j_1, \dots, j_p} \Big( \prod_{i < j} \alpha_{j_i j_j}^{2/(p-1)} \Big) \prod_{s=1}^p \| f_{j_s} \|_p \leq \kappa^{p-1} \sum_j \| f_j \|_p^p.
$$
Thus,
$$\Big\| \sum_j f_j \Big\|_p^p \leq \kappa^{p-1} \sum_j \| f_j \|_p^p,$$
and taking $p$-th roots yields
$$
\Big\| \sum_j f_j \Big\|_p \leq \kappa^{(p-1)/p} \Big( \sum_j \| f_j \|_p^p \Big)^{1/p} = \Big( \sup_j \sum_k \alpha_{j k}^{p'} \Big)^{1/p'} \Big( \sum_j \| f_j \|_p^p \Big)^{1/p}.
$$
This completes the proof of Theorem~\ref{jose}.

\subsection{Proof of  Theorems \ref{absnon} and part of Theorem \ref{mth2}}\label{lieb1}
We will need several technical lemmas.

\begin{lemma}\label{conc1}
For any $p\geq 2$, and all $n>2$ the map 
\begin{align*}
    \varphi_{n,p}(t) = \Big(1+(n-1) \Big( t n \binom{n}{2}^{-p/2} \Big)^{\beta(n,p)}\Big)^{p-1}
\end{align*}
is concave on $\left[0,\frac{1}{n} \binom{n}{2}^{p/2}\right]$. 
\end{lemma}
\begin{remark}
The proof of the lemma presented below is due to Grok. It is correct and we present it here in a slightly polished form. For the full chat link conversation see Appendix~\ref{grokchat}. 
\end{remark}
\begin{proof}
Let $\beta:=\beta(n,p)$. By rescaling it suffices to show that $\psi(t) = (1+(n-1)t^{\beta})^{p-1}$ is concave on $[0,1]$. Set $u=n-1>1$ and $v=p-1\geq 1$. We have $\psi(t) = (1+ut^{\beta})^{v}$, so 
\begin{align*}
    \psi'(t) &= vu\beta t^{\beta-1}(1+ut^{\beta})^{v-1};\\
    \psi''(t) &= vu\beta (\beta-1) t^{\beta-2}(1+ut^{\beta})^{v-1}+v(v-1)u^{2}\beta^{2} t^{2(\beta-1)}(1+ut^{\beta})^{v-2}\\
    &=uv\beta t^{\beta-2}(1+ut^{\beta})^{v-2}[(\beta-1)(1+ut^{\beta})+u(v-1)\beta t^{\beta}]\\
    &=uv\beta t^{\beta-2}(1+ut^{\beta})^{v-2} [\beta -1 +ut^{\beta}(v\beta-1)].
\end{align*}

Thus the sign of $\psi''$ is determined by the sign of $h(t) := \beta -1 +ut^{\beta}(v\beta-1)$. Next, we claim 
\begin{align}\label{dax1}
    \beta \leq \frac{n}{1+(n-1)(p-1)}. 
\end{align}
Assuming the claim, notice that $\frac{n}{1+(n-1)(p-1)}\leq 1$. Therefore $h(0) = \beta-1\leq 0$. Since $h$ is monotone it suffices to verify $h(1)\leq 0$. We have $h(1) = \beta-1 + uv\beta-u$. Therefore $h(1)\leq 0$ if and only if $\beta \leq \frac{u+1}{1+uv}$. This inequality coincides with \eqref{dax1}. 

To verify \eqref{dax1} we show that $\frac{1}{\beta} \geq \frac{1+(n-1)(p-1)}{n}$. Substituting the expression for $\beta$ yields 
\begin{align*}
    \frac{\frac{p}{2}\ln\left(1+\frac{2}{n-2} \right) + \ln\left(\frac{n-1}{n}\right)}{\ln\left(1+ \frac{1}{n-2}\right)} \geq \frac{1+(n-1)(p-1)}{n}.
\end{align*}
Both sides are affine in $p$, and equal at $p=2$.  The coefficient of $p$ on the left is $\frac{1}{2} \frac{ \ln\left(1 + \frac{2}{n-2}\right) }{ \ln\left(1 + \frac{1}{n-2}\right) }$, and on the right is $\frac{n-1}{n}$.

Set $z = \frac{1}{n-2} > 0$. The inequality on coefficients simplifies to $\frac{ \ln(1 + 2z) }{ \ln(1 + z) } \geq 2 \frac{1 + z}{1 + 2z}$. Define $m(z) = \ln(1 + 2z) - 2 \frac{1 + z}{1 + 2z} \ln(1 + z)$. Then $m(0) = 0$ and $m'(z) = \frac{2 \ln(1 + z)}{(1 + 2z)^2} > 0$ for $z > 0$, so $m(z) > 0$ for $z > 0$. Thus, the coefficient on the left exceeds that on the right for $p > 2$, implying the desired inequality for $p \geq 2$.
\end{proof}


\begin{lemma}\label{teknika1}
    Let $p>2$, and let $n>2$ be an integer. The following are equivalent
\begin{itemize}
    \item[\textup{(i)}] The inequality   \begin{align}\label{utoloba1}
        \Big(\sum_{j=1}^{n} x_{j}\Big)^{p} \leq  \varphi_{n,p}\Big(\Big(\sum_{1\leq j<k\leq n} x_{j}x_{k}\Big)^{p/2}\Big) 
    \end{align}
    holds for all $x_{1}, \ldots, x_{n}\geq 0$ satisfying $\sum_{j} x_{j}^{p}=1$
    \item[\textup{(ii)}] The inequality 
    \begin{align*}
    (kx+(m-k)y)^{p} \leq \varphi_{n,p} \left((x^{2}k(k-1)/2+y^{2}(m-k)(m-k-1)/2+k(m-k)xy)^{p/2}\right) 
\end{align*}
holds for all $x,y\geq 0$, and all integers $k=0, \ldots, m$, and $m=1, \ldots, n$  satisfying $kx^{p}+(m-k)y^{p}=1$.
\end{itemize}

\end{lemma}

Before we proceed with the proof of the lemma notice that the inequality \eqref{utoloba1} is the same as the one in Theorem~\ref{absnon}. 

    \begin{proof}
The implication $\textup{(i)} \Rightarrow \textup{(ii)}$ is trivial, take $x_{1}=x_{2}=\ldots=x_{k}=x$, $x_{k+1}=\ldots=x_{m}=y$, and the rest to be zero. So in what follows we focus on the implication $\textup{(ii)} \Rightarrow \textup{(i)}$.

      Clearly $\varphi_{n,p}$ is increasing.  Therefore it suffices to show that $T^{p}\leq \varphi_{n,p}(b^{p/2})$ where $T = \sup \sum_{j}x_{j}$, where the supremum is taken over all nonnegative variables $x_{1}, \ldots, x_{n}\geq 0$ satisfying the constraints $\sum_{j=1}^{n}x_{j}^{p}=1$ and $\sum_{1\leq i<j\leq n} x_{i}x_{j} =b$. Our goal is to understand on which variables $x_{1}, \ldots, x_{n}$ the value $T$ is achieved. We claim that this happens when
\begin{align}\label{veceq}
(x_{1}, \ldots, x_{n}) = (\underbrace{x,\ldots, x}_{k}, \underbrace{y, \ldots, y}_{m-k}, \underbrace{0, \ldots, 0}_{n-m})
\end{align}
for some $x,y>0$ and some integers $k, m$ satisfying $0\leq k \leq m \leq n$. We should point out that equality in \eqref{veceq} should be considered up to permutations of coordinates in the right-hand side of \eqref{veceq}, i.e., the equality \eqref{veceq} simply says that $k$ of the variables $x_{1}, \ldots, x_{n}$ equal to $x$; $m-k$ of them equal to $y$, and the rest is zero. 

First notice that $b \in [0, \binom{n}{2} n^{-2/p}]$. Indeed, the value $b=0$ is achieved if and only if $x_{k}=1$ for some integer $k$, and $x_{j}=0$ for all $j\neq k$. On the other hand by Lagrange multipliers the local maximum of $\sum_{i<j}x_{i}x_{j}$ under the constraint $\sum_{j} x_{j}^{p}=1$, $x_{j}>0$ for all $j=1, \ldots, n$  is achieved when $\sum_{j\neq k}x_{j} - \lambda px_{k}^{p-1}=0$ for all $k=1, \ldots, n$, and some $\lambda>0$ (notice that $\lambda\leq 0$ does not have a solution unless $x_{1}=\ldots =x_{n}=0$ which  would violate $\sum_{j} x_{j}^{p}=1$). Thus $s:=\sum_{j=1}^{n}x_{j} = \lambda p x_{k}^{p-1}+x_{k}$. Notice that the function $h(t) = \lambda p t^{p-1}+t$ is strictly increasing on $[0, \infty)$ implying that  $\lambda p x_{k}^{p-1}+x_{k} =\lambda px_{j}^{p-1}+x_{j}$ can happen for all pairs $j,k$  if and only if $x_{1}=\ldots = x_{n}=x$. In this case $1=\sum_{j=1}^{n}x_{j}^{p} = nx^{p}$, and $\sum_{1\leq i<j\leq n} x_{i}x_{j} = \binom{n}{2}x^{2} = \binom{n}{2}n^{-2/p}$. 
      
      The value  $\binom{n}{2}n^{-2/p}$ is also a global maximum. Indeed, on the boundary where some of the $x_{j}$'s vanish we can argue in a similar way and conclude that the rest of the variables have to be equal to each other. This leads to $\sum_{1\leq i<j\leq n}x_{i}x_{j} = \binom{k}{2}k^{-2/p}$ for some integer $k$, $1\leq k \leq n$. On the other hand the map $k \mapsto \binom{k}{2}k^{-2/p}$ is increasing on $[1, n]$, hence its maximal value is $\binom{n}{2}n^{-2/p}$.

\begin{remark}An alternative way: 
 By AM-GM inequality and H\"older inequality we have
      $$
b\leq \frac{n-1}{2}\sum_{i=1}^{n}x^2_i\leq \frac{n-1}{2}\Big(\sum_{i=1}^{n}x^p_i\Big)^{2/p}n^{1-2/p}=\binom{n}{2}n^{-2/p}.
      $$
      The equality happens if and only if $x_1=x_2=\dots=x_n=\frac{1}{n^{1/p}}$.
\end{remark}

      Thus the implication $\textup{(ii)} \Rightarrow \textup{(i)}$ is trivial if $b=0$ or $b=\binom{n}{2}n^{-2/p}$. In what follows assume $b\in (0, \binom{n}{2}n^{-2/p})$. Let the structure of local maximum of $\sum_{j} x_{j}$ in the interior of the domain, i.e., $x_{1}, \ldots, x_{n}>0$ and later we will explain what happens when part of the $x_{j}$ are equal to zero.

      Consider Lagrangian function
      \begin{align*}
          \Psi(x_{1}, \ldots, x_{n}) = \sum_{j} x_{j} - \lambda \Big(\sum_{j} x_{j}^{p} -1\Big) - \mu \Big(\sum_{j<k} x_{j}x_{k} -b\Big).
      \end{align*}
      Partial derivatives of $\Psi$  vanish if and only if $1-\lambda p x_{k}^{p-1} - \mu \sum_{j\neq k} x_{j}=0$ for all $k=1, \ldots, n$. Adding $-\mu x_{k}$ to both sides of the equation,  and denoting  as before $s := \sum_{j=1}^{n}x_{j}$ we obtain 
      \begin{align*}
          1-\mu s = \lambda p x_{k}^{p-1} - \mu x_{k}\quad \text{for all} \quad k=1, \ldots, n. 
      \end{align*}

      Let $g(t) = \lambda p t^{p-1} - \mu t$. If $\lambda=0$ then the conclusion is trivial. If $\lambda \neq 0$ then since either $g''>0$ or $g''<0$ the graph of $g(t)$ intersects any horizontal line in at most two points. This shows that at the interior local maximum, the variables $x_{1}, \ldots, x_{n}$ take at most two values. 

      If we are on the boundary, say part of the $x_{j}$ are zero, then we can argue similarly as before. So we conclude that at the global maximum variables $x_{1}, \ldots, x_{n}$ take at most 3 values, these are $0,u,v$ for some $u,v>0$. 
    \end{proof}

Now we are ready to prove Theorem~\ref{absnon}. 

\begin{proof}[Proof of Theorem~\ref{absnon}]
The implication \eqref{CFL1} with $r=\beta(n,p)$ implies \eqref{num1} is trivial, just take $f_{j} = a_{j} \mathbf{1}_{A}, j=1, \ldots, n$, where $A$ is a common measurable set with finite positive measure. So in what follows we focus on the implication \eqref{num1} implies \eqref{CFL1}. 

Recall that the inequality \eqref{utoloba1} is the same as  \eqref{num1}. Pick nonnegative functions $f_{1},\ldots, f_{n}$ in $L^{p}$. Then the validity of \eqref{utoloba1} implies 
\begin{align}\label{ut22}
    \Big(\sum_{j=1}^{n} g_{j} \Big)^{p} \leq \varphi_{n,p}\Big(\Big(\sum_{1\leq i<j\leq n} g_{i} g_{j}\Big)^{p/2} \Big)
\end{align}
on a set $X' = \{ \sum_{j} f^{p}_{j}>0\}$, where $g_{k} = \frac{f_{k}}{\big(\sum_{j} f_{j}^{p}\big)^{1/p}}$, $k=1,\ldots,n$.  Consider the probability measure 
\begin{align*}
    d\nu = \frac{\sum_{j} f_{j}^{p}}{\sum_{j} \|f_{j}\|_{p}^{p}}\, d\mu
\end{align*}
on a set $X'$. Recall that the function $\varphi_{n,p}$ is concave on $[0, \frac{1}{n} \binom{n}{2}^{p/2}]$, and the quantity $\Big(\sum_{1\leq j<i\leq n}g_{i}g_{j} \Big)^{p/2}$ is at most $\frac{1}{n} \binom{n}{2}^{p/2}$ (see the proof of Lemma~\ref{teknika1}). So, integrating the inequality \eqref{ut22} over the set $X'$ with respect to the probability measure $d\nu$ and applying Jensen's inequality (here we invoke Lemma~\ref{conc1}) we obtain 
\begin{align}\label{ut33}
    \int \Big(\sum_{j} g_{j} \Big)^{p} d\nu   \leq \varphi_{n,p}\Big( \int \Big(\sum_{i<j} g_{i} g_{j}\Big)^{p/2} d\nu \Big).
\end{align}
 It is a straightforward calculation to verify that the inequality \eqref{ut33} coincides with \eqref{CFL1}  $r=\beta(n,p)$.  This finishes the proof of Theorem~\ref{absnon}.
\end{proof}

\section{Completing the proof of Theorem~\ref{mth2}}\label{mivedi}

Notice an identity that is only valid for $n=3$:
\begin{align}\label{eq: nice identity for n=3}
    (n-1) \left( \frac{n}{2}\cdot \binom{n}{2}^{-p/2}\right)^{\beta(n,p)}=1.
\end{align}
Therefore, in this special case we have $\varphi_{3,p}(t) = \Big(1+(2t)^{\beta(3,p)}\Big)^{p-1}$. Thus the inequality \eqref{utoloba1}, after raising both sides to the power $1/(p-1)$, can be rewritten as follows 
\begin{align*}
(x+y+z)^{p'} \leq 1+\Big(2(xy+xz+yz)^{p/2}\Big)^{\beta(3,p)}.
\end{align*}

\begin{lemma}[main inequality]\label{axali1}
Let $p\geq 3$, and let $c(p) = \frac{2\ln(2)}{(p-2)\ln(3)+2\ln(2)}$. The inequality 
\begin{align}\label{sami1}
(x+y+z)^{p'} \leq 1+\Big(2(xy+xz+yz)^{p/2}\Big)^{c(p)}  
\end{align}
holds for all $x,y,z\geq 0$ satisfying $x^{p}+y^{p}+z^{p}=1$.
\end{lemma}

In what follows by continuity we assume $p>3$.

\medskip
\noindent\textbf{Step 1}. First we verify that the inequality holds on the boundary, i.e., when at least one of the variables is zero. Without loss of generality assume $z=0$. In this case the inequality takes the form 
\begin{align}\label{tej1}
    (x+y)^{p'}\leq 1+(2x^{p/2}y^{p/2})^{c(p)}.
\end{align}
Notice that $c(p) \leq \frac{2}{p}$. Indeed, to verify the inequality $\frac{2\ln(2)}{(p-2)\ln(3)+2\ln(2)} \leq \frac{2}{p}$ divide both sides by $2$ and multiply by the denominators, we arrive at $p \ln(2) \leq (p-2)\ln(3)+2\ln(2)$ which follows from $(p-2)(\ln(3)-\ln(2))\geq 0$. Next we have $2x^{p/2}y^{p/2}\leq x^{p}+y^{p}=1$, and hence the right hand side of \eqref{tej1} is at least $1+(2x^{p/2}y^{p/2})^{2/p} = 1+2^{2/p}xy$. On the other hand the stronger inequality 
\begin{align*}
    (x+y)^{p'}\leq 1+2^{2/p}xy
\end{align*}
under the constraint $x^{p}+y^{p}=1$ was proved recently in Theorem~1.1 (case $p\geq 2$) in \cite{CFIL2021}.

\medskip
\noindent\textbf{Step 2.} By Lemma~\ref{teknika1} applied to $n=3$ and Step 1 it suffices to verify the inequality \eqref{sami1} when in the variables $x,y$ and $z$ are positive and two of them are equal to each other. Without loss of generality we can assume that $z=y$. So we want to show 
\begin{align}\label{sof1}
    (x+2y)^{p'}\leq 1+2^{c}(2xy+y^{2})^{pc/2}
\end{align}
under the constraint $x^{p}+2y^{p}=1$ where $c=c(p)$. Set $t := y/x$. From $x^{p}+2y^{p}=1$ we obtain $x^{p}(1+2t^{p})=1$ and hence $x= (1+2t^{p})^{-1/p}$. 

Next, we have $x+2y=x(1+2t)=(1+2t^{p})^{-1/p} (1+2t)$, therefore, the left hand side of \eqref{sof1} takes the form 
\begin{align*}
(x+2y)^{p'}=(1+2t^{p})^{-1/(p-1)}(1+2t)^{p'}.
\end{align*}
Similarly we have $2xy+y^{2} = x^{2}(2t+t^{2}) = (1+2t^{p})^{-2/p}(2t+t^{2})$. So the right hand side of \eqref{sof1} takes the form 
\begin{align*}
1+2^{c}(2xy+y^{2})^{pc/2} =1+2^{c}(1+2t^{p})^{-c}(2t+t^{2})^{pc/2}.
\end{align*}

Thus to prove \eqref{sof1} it suffices to show 
\begin{align}\label{side1}
2^{c} \frac{(2t+t^{2})^{pc/2}}{(1+2t^{p})^{c}} - \frac{(1+2t)^{p/(p-1)}}{(1+2t^{p})^{1/(p-1)}} +1\geq 0
\end{align}
for all $t\geq 0$. 
Next we make change of variables $t=1/s$, $s>0$. The expression in the left hand side of  \eqref{side1} rewrites as follows 
\begin{align*}
h(s):=\left( \frac{2(2s+1)^{p/2}}{s^p + 2} \right)^c - \left( \frac{(s+2)^p}{s^p + 2} \right)^{\frac{1}{p-1}} + 1.
\end{align*}

\begin{figure}[H]
\centering
\includegraphics[width=1\textwidth]{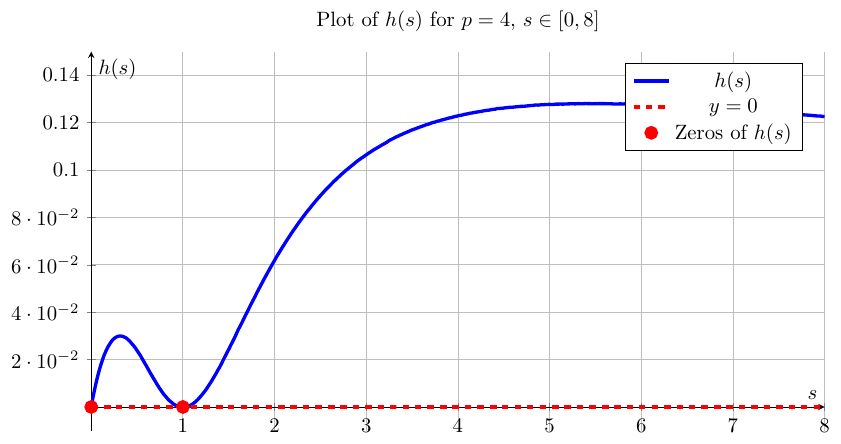}
\label{fig:p1}
\end{figure}

We have 
\begin{align*}
    h'(s) = \dfrac{p}{s^{p} + 2} \left[ -c \left( \dfrac{2(2s+1)^{p/2}}{s^{p}+2} \right)^{c} \dfrac{ s^{p} + s^{p-1} - 2 }{2s + 1} 
+ 2 \left( \dfrac{(s+2)^{p}}{s^{p}+2} \right)^{\frac{1}{p-1}} \dfrac{ s^{p-1} - 1 }{ (p-1)(s+2) } \right].
\end{align*}

\begin{figure}[H]
\centering
\includegraphics[width=1\textwidth]{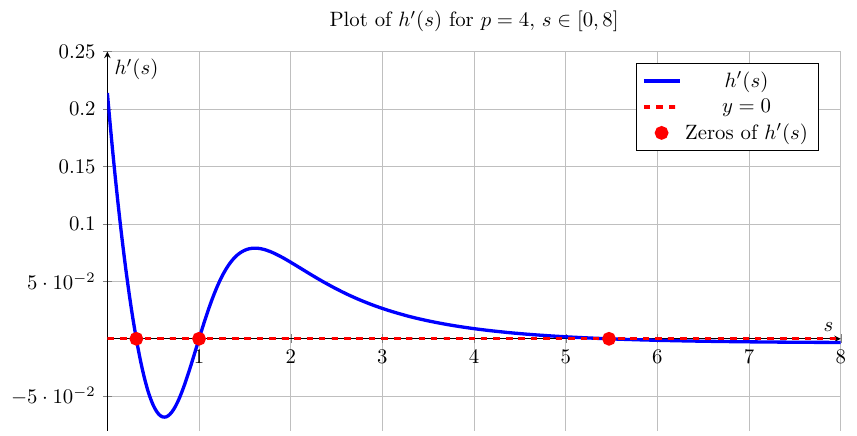}
\label{fig:p2}
\end{figure}

\begin{lemma}\label{tek-1}
    We have $h(0)=h(1)=h'(1)=0$,  $h''(1)= \frac{p(p-2)(3\ln 3 - 4\ln 2)c}{9\ln 2} \geq 0$, and $\lim_{s \to \infty}h(s)=0$.
\end{lemma}

\begin{proof}
 See Appendix~\ref{appendix:A}.
\end{proof}

\begin{lemma}\label{sign-1}
For each $p >3$ function $h'$ changes sign 3 times as follows: 
$$
(+-+-).
$$
More precisely from $+$ to $-$ at point $q_{1} \in (0,1)$; from $-$ to $+$ at point $1$; and from $+$ to $-$ at point $q_{2}>1$. 
\end{lemma}
Clearly Lemma~\ref{sign-1} and Lemma~\ref{tek-1} imply $h\geq 0$ on $[0, \infty)$ and hence the inequality \eqref{side1}. 

\vskip1cm

The proof of Lemma~\ref{sign-1} will be divided into several parts.

\vskip1cm

It is enough to study the sign of the expression 
\begin{align}\label{fac-01}
     -c \left( \dfrac{2(2s+1)^{p/2}}{s^{p}+2} \right)^{c} \dfrac{ s^{p} + s^{p-1} - 2 }{2s + 1} 
+ 2 \left( \dfrac{(s+2)^{p}}{s^{p}+2} \right)^{\frac{1}{p-1}} \dfrac{ s^{p-1} - 1 }{ (p-1)(s+2)}.
\end{align}

The fact that $h'(s)$ changes sign from $-$ to $+$ follows from $h''(1)>0$. 

Multiply the expression \eqref{fac-01} by the factor 
\begin{align}\label{fak21}
    \frac{(p-1)}{2} \frac{2s+1}{s^{p}+s^{p-1}-2} \left( \frac{s^{p}+2}{(2s+1)^{p/2}}\right)^{c}
\end{align}
Then the resulting expression takes the form 
\begin{align}\label{red1}
    -2^{c-1}c\cdot (p-1) + 
\frac{(s+2)^{\frac{1}{p-1}}\cdot\bigl(s^{p-1}-1\bigr)\cdot(2s+1)^{\,1-\frac{pc}{2}}\cdot\bigl(s^{p}+2\bigr)^{\,c-\frac{1}{p-1}}}{s^{p}+s^{p-1}-2}
\end{align}

Since the factor \eqref{fak21}  has sign $-$ on $(0,1)$, and sign $+$ on $(1, \infty)$ therefore the lemma is reduced to showing that the expression \eqref{red1} changes sign two times, namely, from $-$ to $+$ at point $q_{1} \in (0,1)$, and from $+$ to $-$ at point $q_{2}>1$. As $s \mapsto \log(s)$ is increasing on $(0, \infty)$  we can replace the expression \eqref{red1} by 
\begin{align}\label{red2}
   \psi(s):= -\log\left(2^{c-1}c\cdot (p-1) \right) + 
\log\left(\frac{(s+2)^{\frac{1}{p-1}}\cdot\bigl(s^{p-1}-1\bigr)\cdot(2s+1)^{\,1-\frac{pc}{2}}\cdot\bigl(s^{p}+2\bigr)^{\,c-\frac{1}{p-1}}}{s^{p}+s^{p-1}-2} \right),
\end{align}
and study its sign change.

\begin{figure}[H]
\centering
\includegraphics[width=1\textwidth]{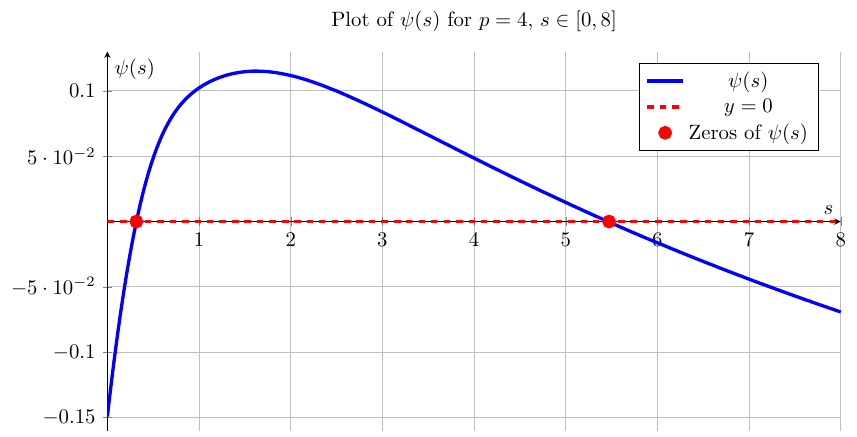}
\label{fig:p3}
\end{figure}

\begin{lemma}\label{tex01}
    We have $\psi(0)=-\log((p-1)c)<0$, $\psi(1) = \log\left(\frac{3}{(2p-1)c}\right)>0$, and $\lim_{s \to \infty} \psi(s) =-\infty$
\end{lemma}
\begin{proof}
 See Appendix~\ref{appendix:A}.
\end{proof}

\vskip1cm 

Thus to show that $\psi$ changes sign from $-$ to $+$ at $q_{1} \in (0,1)$, and from $+$ to $-$ at point $q_{2}>1$ it follows from Lemma~\ref{tex01} that one needs to show  $\psi'$ changes sign once from $+$ to $-$ on $(0, \infty)$. 

A direct differentiation shows that 
\begin{align*}
    \psi'(s) &=\frac{1}{(p-1)(s+2)}
+\frac{(p-1)s^{p-2}}{s^{p-1}-1}
+\frac{2-pc}{1+2s}
+\left(cp-\frac{p}{p-1}\right)\frac{s^{p-1}}{s^p+2}
-\frac{p s^{p-1}+(p-1)s^{p-2}}{s^p+s^{p-1}-2}\,\\
&= -\frac{P(s)}{(p-1)\,(s+2)\,(2s+1)\,\bigl(s^{p-1}-1\bigr)\,\bigl(s^{p}+2\bigr)\,\bigl(s^{p}+s^{p-1}-2\bigr)}, 
\end{align*}
where 
\begin{align*}
P(s):={}&(p-1)(2-cp)\,s^{3p}
+(-4cp^{2}+4cp+6p-2)\,s^{3p-1}
+(-5cp^{2}+5cp+5p+1)\,s^{3p-2} \\
&-2p(cp-c-1)\,s^{3p-3}
+p(c-2)(p-1)\,s^{2p+1}
+(8cp^{2}-8cp-3p^{2}-5)\,s^{2p} \\
&+(17cp^{2}-17cp+3p^{2}-15p-8)\,s^{2p-1}
+(10cp^{2}-10cp+2p^{2}-14p+4)\,s^{2p-2} \\
&-4p(cp-c+p-2)\,s^{p+1}
+(-16cp^{2}+16cp-6p^{2}+24p+4)\,s^{p} \\
&+(-16cp^{2}+16cp+6p^{2}+12p-8)\,s^{p-1}
+4(p-1)^{2}\,s^{p-2}
+4p(cp-c-2)\,s \\
&+(8cp^{2}-8cp-16p+12).
\end{align*}
Clearly $\mathrm{sign}(\psi') = - \mathrm{sign}(P(s))$. So the goal reduces to showing that $P$ changes sign once from $-$ to $+$ on $(0, \infty)$.

\begin{figure}[H]
\centering
\includegraphics[width=1\textwidth]{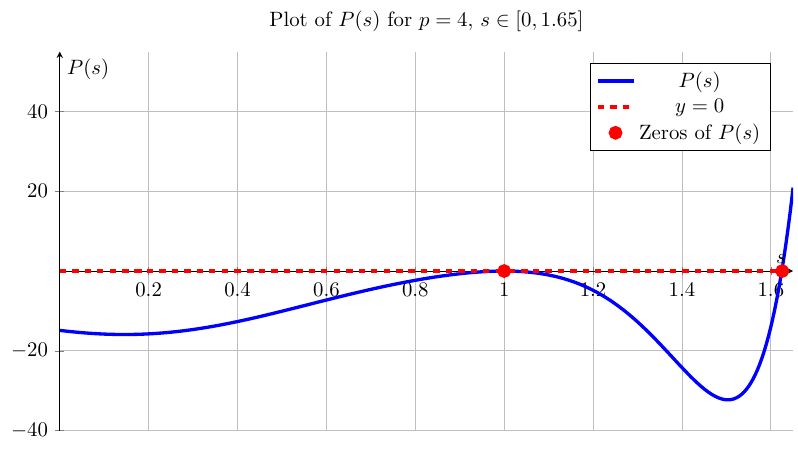}
\label{fig:p4}
\end{figure}

\begin{lemma}\label{lem: P(0), P'(0) etc}
For $p\geq 3$ we have $P(0)<0$, $P'(0)<0$, $P(1)=P'(1)=0$, $P''(1)<0$, and $\lim_{s \to \infty} P(s)=+\infty$. 
\end{lemma}
\begin{proof}
 See Appendix~\ref{appendix:A}.
\end{proof}

Thus to obtain a sign change from $-$ to $+$ for $P$ it suffices to show  (thanks to the previous lemma) the following lemma 
\begin{lemma}\label{mm1}
$P''(s)$ changes sign as $(+-+)$ on $(0, \infty)$. 
\end{lemma}

The proof of Lemma~\ref{mm1} will be split into several parts. 

\vskip0.3cm

First notice that we have  $P''(s) = s^{p-4}Q(s)$, where 

\begin{align*}
Q(s):={}&\,3p(p-1)(3p-1)(2-cp)\,s^{2p+2}
+(-4cp^{2}+4cp+6p-2)(3p-1)(3p-2)\,s^{2p+1} \\
&+(-5cp^{2}+5cp+5p+1)(3p-2)(3p-3)\,s^{2p}
-2p(3p-3)(3p-4)(cp-c-1)\,s^{2p-1} \\
&+p(c-2)(p-1)(2p+1)\,2p\,s^{p+3}
+(8cp^{2}-8cp-3p^{2}-5)\,2p(2p-1)\,s^{p+2} \\
&+(17cp^{2}-17cp+3p^{2}-15p-8)(2p-1)(2p-2)\,s^{p+1} \\
&+(10cp^{2}-10cp+2p^{2}-14p+4)(2p-2)(2p-3)\,s^{p} \\
&-4p(cp-c+p-2)(p+1)p\,s^{3}
+(-16cp^{2}+16cp-6p^{2}+24p+4)\,p(p-1)\,s^{2} \\
&+(-16cp^{2}+16cp+6p^{2}+12p-8)(p-1)(p-2)\,s \\
&+4(p-1)^{2}(p-2)(p-3).
\end{align*}

\begin{figure}[H]
\centering
\includegraphics[width=1\textwidth]{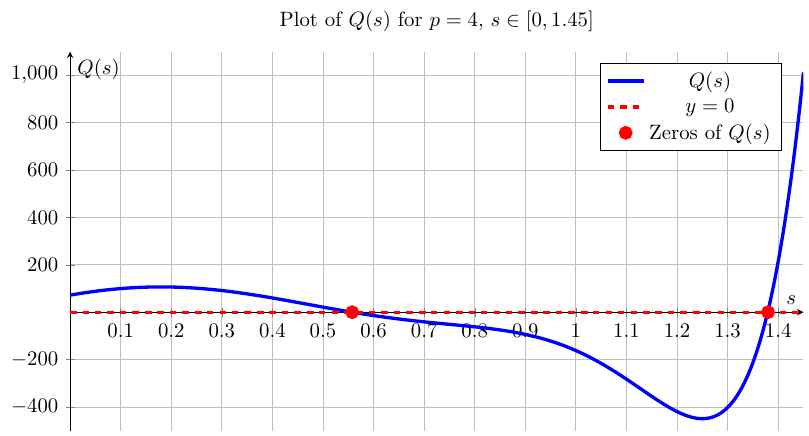}
\label{fig:p5}
\end{figure}

\begin{lemma}\label{dax91}
    We have $Q(0)>0$, and $\lim_{s \to \infty}Q(s)=+\infty$. 
\end{lemma}
\begin{proof}
 See Appendix~\ref{appendix:A}.
\end{proof}

Since $\mathrm{sign}(Q) = \mathrm{sign}(P'')$, invoking Lemma~\ref{dax91} it suffices to show that $Q$ changes sign as $(+-+)$ on $(0, \infty)$. 
\vskip0.3cm 

Notice that $Q(1)=-9(p-1)^{2}(p-2)<0$, thus it suffices to show that $Q'$ has sign change $(+-+)$. 

\vskip0.3cm 

\begin{lemma}\label{lem: Q'(0),Q'(1),Q'(infinity)}
We have $Q'(0)>0$, $Q'(1)<0$, and $Q'(\infty)>0$. 
\end{lemma}
\begin{proof}
 See Appendix~\ref{appendix:A}.
\end{proof}

Thus to conclude that $Q'$ has sign change $(+-+)$ it suffices to show the following key lemma 
\begin{lemma}\label{tek11}
If $Q''(s)=0$ then $Q'(s)<0$.   
\end{lemma}
Clearly if the lemma is true then it implies that $Q'$ has $(+-+)$ sign change. 

\vskip0.5cm

\begin{figure}[H]
\centering
\includegraphics[width=1\textwidth]{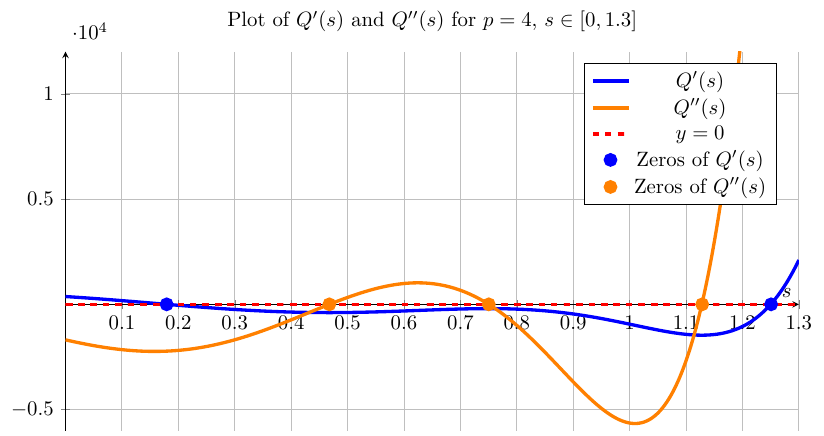}
\label{fig:p6}
\end{figure}

The proof of Lemma~\ref{tek11} will be split into several technical lemmas. 
\begin{lemma}\label{lem: bounds for c}
For any $p>3$ we have 
\begin{align*}
    \frac{5}{4p-3}<c(p)< \frac{4}{3p-2}.
\end{align*}
In particular, we have $\frac{1}{p-1}<c(p)<2/p$. 
\end{lemma} 
\begin{proof}
 See Appendix~\ref{appendix:A}.
\end{proof}

Next let $g(s):=sQ''(s)$, and set $H(s):=sQ''(s)-(p+1)Q'(s)$. 
\begin{remark}
    Notice that the statement ``if $Q''(s)=0$ then $Q'(s)<0$'' is equivalent to the statement ``if $g(s)=0$ for some $s\neq0$ then $H(s)>0$ (or equivalently $Q'(s)<0$)''. 
\end{remark}

Our first goal will be the following. 
\begin{lemma}\label{lem:g-at-most-3}
    For $p>3$ function $g$ has at most 3 positive real zeros (counted with multiplicities). 
\end{lemma}
\begin{proof}
 See Appendix~\ref{appendix:A}.
\end{proof}

\vskip0.3cm 

Our second goal will be to show that $g$ has exactly 3 simple positive  zeros (here simple means that multiplicity is one) with sign pattern
$$
(-+-+).
$$
Indeed, this follows from the previous lemma and the following technical observation. Before we state that observation we denote 
\begin{align*}
    s_{A} = \frac{p-1}{p+1}, \quad r = \frac{p}{p+1} \quad \text{so that} \quad 0<s_{A}<r<1. 
\end{align*}

\begin{lemma}[$g$ changes sign 3 times]\label{lem: g changes sign 3 times}We have

\begin{itemize}
\item[\textup{(i)}]The inequality \begin{align*}
  g(s)<0 \quad \text{holds in a neighborhood of zero;} \quad g(1)<0; \quad \lim_{s \to \infty} g(s) = +\infty.  
\end{align*}
\item[\textup{(ii)}]Moreover, we have\begin{align*}
 \quad g(s_{A})>0; \quad  g(r)<0.  
\end{align*}
\end{itemize}
\end{lemma}
\begin{proof}
 See Appendix~\ref{appendix:A} for the proof of {(i)} and Appendix~\ref{appendix:B} for the proof of {(ii)}.
\end{proof}


  

Clearly these two lemmas combined with the previous lemmas show that $g$ changes sign as $(-+-+)$.

\vskip0.3cm

Next we need to learn how to control the monotonicity of $H$. 

\vskip0.3cm

Let $W= sH'$. Observe that zeros of $W$ and $g$ interlace in the following sense:  between any two positive consecutive  zeros of $g=sQ''(s)$ there is at least one zero of $W$.  This follows from the observation that
$$
W = \left(s\frac{d}{ds}-(p+1)\right)g(s).
$$
Therefore, if we apply the Rolle's theorem to $s^{-(p+1)}g(s)$ we will obtain the desired claim. 

\vskip1cm 
Since $g$ has exactly 3 positive zeros, it follows that $W$ has at least two different positive zeros.  

\begin{lemma}\label{lem: W has two zeros}
$W$ has exactly two positive simple zeros $0<u_{1}<u_{2}$. Moreover, $W>0$ on $(0, u_{1})$; $W<0$ on $(u_{1}, u_{2})$; and $W>0$ on $(u_{2}, \infty)$. 
\end{lemma}

\begin{proof}
 See Appendix~\ref{appendix:A}.
\end{proof}

\begin{lemma}\label{lem: W(r)<0}
We have $W(r)<0$. 
\end{lemma}
\begin{proof}
 See Appendix~\ref{appendix:B}.
\end{proof}

One more technical lemma
\begin{lemma}\label{lem: tecnichal lemma about Q'sA and Hr}
 We have $Q'(s_{A})<0$, and $H(r)>0$.  
\end{lemma}
\begin{proof}
 See Appendix~\ref{appendix:B}.
\end{proof}

Finally we are ready to show the complete proof of Lemma~\ref{tek11}. Here is its reformulation: 

\begin{lemma}
For every $p>3$ if $g(s)=0$ (for some $s>0$) then $H(s)>0$ (equivalently $f:=Q'(s)<0$).
\end{lemma}

\begin{proof}
Fix $p>3$. By previous lemmas $g$ has exactly three simple positive zeros, denote them
\[
0<s_1<s_2<s_3.
\]
Moreover, since $g$ changes sign at each $s_i$, the sign pattern is
\begin{equation}\label{eq:g-sign-pattern}
g<0 \text{ on }(0,s_1),\quad
g>0 \text{ on }(s_1,s_2),\quad
g<0 \text{ on }(s_2,s_3),\quad
g>0 \text{ on }(s_3,\infty).
\end{equation}
Since $g(s)=sf'(s)$ and $s>0$, this is also the sign pattern of $f'(s)$.

\medskip
\noindent\textbf{Step 1:} Recall that $W$ has exactly two positive simple zeros $0<u_{1}<u_{2}$.  We know that

\begin{equation}\label{eq:u1ru2}
0<u_1<r<u_2.
\end{equation}
 and
\begin{equation}\label{eq:W-signpattern}
W>0\text{ on }(0,u_1),\quad
W<0\text{ on }(u_1,u_2),\quad
W>0\text{ on }(u_2,\infty).
\end{equation}

\medskip
\noindent\textbf{Step 2: interlacing $s_1<u_1<s_2<u_2<s_3$.}
This we know because of the Rolle's theorem we mentioned before (between any two positive consecutive zeros of $g$ there is at least one zero of $W$). Thus we obtain
\begin{equation}\label{eq:interlacing}
s_1<u_1<s_2<u_2<s_3.
\end{equation}

\medskip
\noindent\textbf{Step 3: monotonicity of $H$ from $W=sH'$.}
Since $W(s)=sH'(s)$ and $s>0$, $\mathrm{sign}(H')=\mathrm{sign}(W)$.
Thus, \eqref{eq:W-signpattern} implies
\begin{align*}
&H \text{ is strictly increasing on }(0,u_1),\quad
H \text{ is strictly decreasing on }(u_1,u_2),\quad\\
&H \text{ is strictly increasing on }(u_2,\infty).
\end{align*}
(This is the key reason for introducing $W$, it will help us to show $H(s_{2})>0$). 

\medskip

\noindent\textbf{Step 4: prove $H(s_i)>0$ for $i=1,2,3$.}

\smallskip
\noindent\emph{(i) The middle root $s_2$.}
By \eqref{eq:interlacing} and \eqref{eq:u1ru2} we have $u_1<s_2<r<u_2$ (to get this order we are using that both $W(r)$ and $g(r)$ are negative).
Since $H$ is decreasing on $(u_1,u_2)$, we obtain
$$
H(s_2)>H(r)>0.
$$

\smallskip
\noindent\emph{(ii) The first root $s_1$.}
From \eqref{eq:g-sign-pattern}, $f'(s)<0$ on $(0,s_1)$ and $f'(s)>0$ on $(s_1,s_2)$,
so $s_1$ is a strict local minimum of $f$, and $f$ is strictly increasing on $(s_1,s_2)$.
Also recall that $s_{A}<r$, $g(s_{A})>0$, and $g(r)<0$. Therefore  $s_A\in(s_1,s_2)$.
Hence
$$
f(s_1)<f(s_A)=Q'(s_{A})<0.
$$

\smallskip
\noindent\emph{(iii) The third root $s_3$.}
Again, from \eqref{eq:g-sign-pattern}, $f'(s)<0$ on $(s_2,s_3)$ and $f'(s)>0$ on $(s_3,\infty)$,
so $s_3$ is a strict local minimum of $f$, and $f$ is strictly decreasing on $(s_2,s_3)$.
Since $g(r)<0$ and $r<1$, while $g(1)<0$ and $g(s)\to+\infty$ as $s\to\infty$, the third zero satisfies $s_3>1$
(indeed $g$ is negative at $1$ but eventually positive).
Thus, $1\in(s_2,s_3)$ and monotonicity on $(s_2,s_3)$ yields
$$
f(s_3)<f(1)=Q'(1)<0.
$$
This finishes the proof of the lemma.
\end{proof}

\pgfmathsetmacro{\cc}{ln(2)/ln(6)}

\begin{figure}[H]
\centering
\includegraphics[width=1\textwidth]{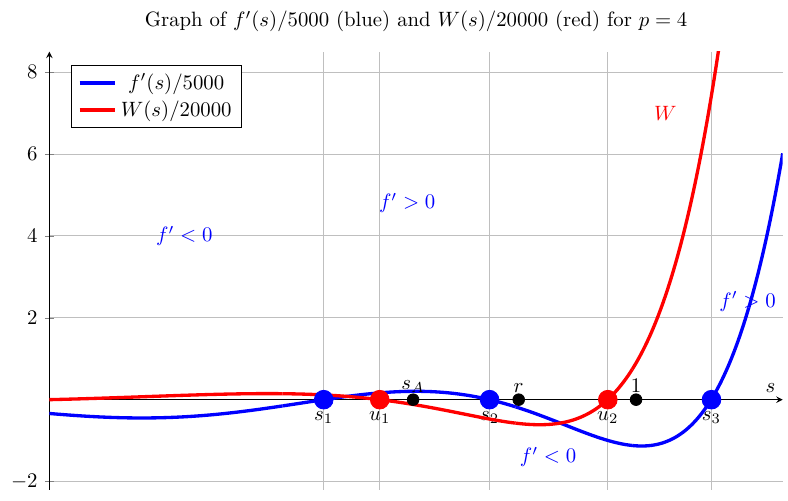}
\label{fig:p7}
\end{figure}

So to summarize, we needed the following list of technical inequalities whose proofs are given in appendices A and B: $Q'(1)<0; Q'(s_{A})<0; g(s_{A})>0; g(r)<0; g(1)<0; W(r)<0; H(r) = rQ''(r)-(p+1)Q'(r)>0;$ sign pattern of coefficients for $g$  and sign pattern for $W$.

\appendix
\section{An application of Descartes' generalized rule of signs}
\label{appendix:A}

\begin{proof}[Proof of Lemma~\ref{tek-1}]
We first show $h(0)=0$. Substituting $s=0$ into the definition of $h(s)$ gives
\[
\left( \frac{2(2\cdot0+1)^{p/2}}{0+2} \right)^c - \left( \frac{(0+2)^p}{0+2} \right)^{\frac{1}{p-1}} + 1 = 1^c - 2 + 1 = 0.
\]

Next, $h(1)=0$. Substituting $s=1$ yields
\[
\left( \frac{2\cdot3^{p/2}}{3} \right)^c - \left( \frac{3^p}{3} \right)^{\frac{1}{p-1}} + 1 = \left(2 \cdot 3^{(p-2)/2}\right)^c - 3 + 1.
\]
By the choice of $c=c(p)$,
\[
c \cdot \ln\bigl(2 \cdot 3^{(p-2)/2}\bigr) = \ln 2,
\]
so $\bigl(2 \cdot 3^{(p-2)/2}\bigr)^c = 2$. Therefore $h(1) = 2 - 3 + 1 = 0$.

For $h'(1)=0$, note that the prefactor $p/(s^p + 2)$ in the given expression for $h'(s)$ evaluates to $p/3 > 0$ at $s=1$. The bracket, however, consists of two summands: the first contains the factor $s^p + s^{p-1} - 2$, which vanishes at $s=1$; the second contains the factor $s^{p-1} - 1$, which also vanishes at $s=1$. Hence the bracket is zero at $s=1$, and $h'(1)=0$.

To obtain $h''(1)$, write
\[
h'(s) = \frac{p}{s^p + 2} \, M(s),
\]
where $M(s)$ denotes the bracket in the given formula for $h'(s)$. Since $M(1)=0$,
\[
h''(1) = \frac{p}{3} M'(1).
\]
Denote
\[
F(s) = \left( \frac{2(2s+1)^{p/2}}{s^p + 2} \right)^c, \quad G(s) = \left( \frac{(s+2)^p}{s^p + 2} \right)^{\frac{1}{p-1}}.
\]
Then
\[
M(s) = -c \, F(s) \, \frac{s^p + s^{p-1} - 2}{2s + 1} + 2 \, G(s) \, \frac{s^{p-1} - 1}{(p-1)(s+2)}.
\]
We have $F(1)=2$ and $G(1)=3$. Moreover, both fractions vanish at $s=1$. Therefore, when differentiating, the product rule applied to each term yields (at $s=1$)
\[
M'(1) = -c \cdot 2 \cdot \left( \frac{d}{ds} \frac{s^p + s^{p-1} - 2}{2s + 1} \right)\bigg|_{s=1} + 6 \cdot \left( \frac{d}{ds} \frac{s^{p-1} - 1}{(p-1)(s+2)} \right)\bigg|_{s=1}.
\]
For a quotient $\mathrm{num}/\mathrm{den}$ with $\mathrm{num}(1)=0$, the derivative at $s=1$ simplifies to $\mathrm{num}'(1)/\mathrm{den}(1)$.

 For the first fraction: $\mathrm{num}'(s) = p s^{p-1} + (p-1)s^{p-2}$, so $\mathrm{num}'(1) = 2p-1$ and $\mathrm{den}(1)=3$. Thus its derivative at $1$ is $(2p-1)/3$.
 For the second fraction: $\mathrm{num}'(s) = (p-1)s^{p-2}$, so $\mathrm{num}'(1) = p-1$ and $\mathrm{den}(1) = 3(p-1)$. Thus its derivative at $1$ is $1/3$.

Hence
\[
M'(1) = -2c \cdot \frac{2p-1}{3} + 6 \cdot \frac{1}{3} = 2 - \frac{2c(2p-1)}{3}.
\]
It follows that
\[
h''(1) = \frac{p}{3} \left( 2 - \frac{2c(2p-1)}{3} \right) = \frac{p(p-2)(3\ln 3 - 4\ln 2)c}{9\ln 2}.
\]

Finally, for the limit as $s \to \infty$:
\[
\frac{2(2s+1)^{p/2}}{s^p + 2} \sim 2^{1 + p/2} s^{-p/2} \to 0,
\]
so the first summand raised to the power $c$ tends to $0$. For the second summand,
\[
\frac{(s+2)^p}{s^p + 2} \to 1,
\]
hence its $(1/(p-1))$-th power tends to $1$. Therefore $h(s) \to 0 - 1 + 1 = 0$ as $s \to \infty$.
\end{proof}

\begin{proof}[Proof of Lemma \ref{lem: bounds for c}]
Observe that since $2^6<3^4$, then
$8\ln 3-8\ln 2-4\ln 2<3[4\ln 3-6\ln 2]<[4\ln 3-6\ln 2]p$. Thus $c(p)=\frac{2\ln 2}{2\ln 2+(p-2)\ln 3}< \frac{4}{3p-2}$. Similarly, since $3^5<2^8$, then $6\ln 2+10\ln 2-10\ln 3<3[8\ln 2-5\ln 3]<[8\ln 2-5\ln 3]p$, or equivalently $\frac{5}{4p-3}<c(p)$.
\end{proof}

\begin{proof}[Proof of Lemma \ref{tex01}]
We start observing that 
\begin{equation}
    \psi(0)=-\log(2^{c-1}c(p-1))+\log(2^{c-1})=-\log((p-1)c)<0
\end{equation}
since $c>\frac{1}{p-1}$ by Lemma \ref{lem: bounds for c}. By L'Hopital's rule and \eqref{eq: nice identity for n=3} we also have
\begin{align*}
 \psi(1)&=-\log(2^{c-1}c(p-1))+\log\frac{(p-1)3^{c+1-\frac{pc}{2}}}{2p-1}\\
 &=\log\left[\left(\frac{3}{(2p-1)c}\right)\left(2\frac{3^{(1-p/2)c}}{2^c}\right)\right]\\
 &=\log\left(\frac{3}{(2p-1)c}\right)>0,
\end{align*}
where the last inequality follows from Lemma \ref{lem: bounds for c}, since $\frac{3}{2p-1}>\frac{4}{3p-2}>c$. Finally, for the limit as $s \to \infty$, we observe
\begin{align*}
    \log\left(\frac{(s+2)^{\frac{1}{p-1}}\cdot\bigl(s^{p-1}-1\bigr)\cdot(2s+1)^{\,1-\frac{pc}{2}}\cdot\bigl(s^{p}+2\bigr)^{\,c-\frac{1}{p-1}}}{s^{p}+s^{p-1}-2} \right)&\sim\log s^{\frac{1}{p-1}+p-1+1-\frac{pc}{2}+pc-\frac{p}{p-1}-p}\\
    &\sim\left(\frac{pc}{2}-1\right)\log s.
\end{align*}
Then $\lim_{s \to +\infty} \psi(s) =-\infty$, since $c<\frac{2}{p}$ by Lemma \ref{lem: bounds for c}.
\end{proof}

\begin{proof}[Proof of Lemma \ref{lem: P(0), P'(0) etc}]
By Lemma \ref{lem: bounds for c}
\begin{align*}
c<\frac{2}{p}<\frac{4p-3}{2p(p-1)}=\frac{16p-12}{8p^2-8p},    
\end{align*}
then $P(0)=8cp^{2}-8cp-16p+12<0$. Moreover, $P'(0)=4p(cp-c-2)<0$ since $c<\frac{2}{p}<\frac{2}{p-1}$. Direct evaluation also shows that $P(1)=P'(1)=0$, and $P''(1)=-9(p-2)(p-1)^2<0$. Finally, since $c<\frac{2}{p}$, then the leading coefficient of $P$ is positive, thus $\lim_{s \to \infty} P(s)=+\infty$.
\end{proof}

\begin{proof}[Proof of Lemma \ref{dax91}]
Observe that $Q(0)=4(p-1)^2(p-2)(p-3)>0$ since $p>3$. Moreover, since $c<\frac{2}{p}$, the leading coefficient $3p(p-1)(2-cp)$ is positive, thus $\lim_{s \to \infty}Q(s)=+\infty$. 
\end{proof}

\begin{proof}[Proof of Lemma \ref{lem: Q'(0),Q'(1),Q'(infinity)}]
We define 
\begin{equation*}
    F(x):=(\ln 3)(3x^3-16x+8)+(\ln 4)(-5x^2+14x-4), 
\end{equation*}
for $x\in(3,\infty)$. We observe that
\begin{align*}
    F''(x)=18x\ln 3-10\ln 4>54\ln 3-10\ln 4>0, 
\end{align*}
for all $x\in(3,\infty)$. Then $F'$ is increasing, moreover
\begin{align*}
    F'(x)=(\ln 3)(9x^2-16)+(\ln 4)(-10x+14).
\end{align*}
Then, $F'(x)\geq 65\ln 3-16\ln 4>0$ for all $x\in(3,\infty)$. Thus, $F$ is increasing, in particular
\begin{align*}
    F(x)\geq 41\ln 3-7\ln 4>0.
\end{align*}
Therefore $Q'(0)>0$, since
\begin{align*}
    c=\frac{\ln 4}{(p-2)\ln 3+\ln 4}<\frac{3p^2+6p-4}{8p(p-1)}=\frac{6p^2+12p-8}{16p(p-1)}.
\end{align*}
Moreover, by Lemma \ref{lem: bounds for c} we have $c<\frac{2}{p}$, hence
$3p(p-1)(3p-1)(2-cp)>0$. Therefore the leading coefficient of $Q'(s)$ is positive, so
\[
\lim_{s\to\infty} Q'(s)=+\infty.
\]
In particular, $Q'(\infty)>0$.
Finally, direct evaluation shows that 
\begin{align*}
Q'(1)=-18\,(p-1)^2(2p-1)(2cp^2-cp-2p-2)<0,
\end{align*}
since by Lemma \ref{lem: bounds for c} we have $c>\frac{5}{4p-3}>\frac{2(p+1)}{p(2p-1)}$, where in the last inequality we used that $2p^2-7p+6=(2p-3)(p-2)>0$ for all $p\in(3,\infty)$.
\end{proof}

\begin{proof}[Proof of Lemma \ref{lem:g-at-most-3}]
Fix $p>3$ and write $c=c(p)$. Recall (from the bounds on $c(p)$ proved above) that
\begin{equation}\label{eq:c-bounds-used}
\frac{5}{4p-3}<c<\frac{4}{3p-2}
\qquad\text{and in particular}\qquad
\frac{1}{p-1}<c<\frac{2}{p}.
\end{equation}

Starting from the explicit formula for $Q$ written above, a direct differentiation gives that
$g(s)=sQ''(s)$ is a \emph{pseudopolynomial} (generalized polynomial) with the following $10$ monomials:
\begin{align}\label{eq:g-expansion}
&g(s)=A_{1}\,s + A_{2}\,s^{2}
+ A_{p-1}\,s^{p-1}+A_{p}\,s^{p}
+ A_{p+1}\,s^{p+1}+A_{p+2}\,s^{p+2} \nonumber \\
&+ A_{2p-2}\,s^{2p-2}+A_{2p-1}\,s^{2p-1}
+ A_{2p}\,s^{2p}+A_{2p+1}\,s^{2p+1},
\end{align}
where the coefficients are (all prefactors below are strictly positive for $p>3$)
\begin{align*}
A_{1} &=
2p(p-1)\Bigl(-16c\,p(p-1)-6p^{2}+24p+4\Bigr),\\
A_{2} &=
-24\,p^{2}(p+1)\Bigl(c(p-1)+p-2\Bigr),\\
A_{p-1} &=
p(p-1)(2p-2)(2p-3)\Bigl(10c\,p(p-1)+2p^{2}-14p+4\Bigr),\\
A_{p} &=
p(p+1)(2p-1)(2p-2)\Bigl(17c\,p(p-1)+3p^{2}-15p-8\Bigr),\\
A_{p+1} &=
2p(2p-1)(p+1)(p+2)\Bigl(8c\,p(p-1)-3p^{2}-5\Bigr),\\
A_{p+2} &=
2p^{2}(p-1)(p+2)(p+3)(2p+1)\,(c-2),\\
A_{2p-2} &=
-2p(2p-1)(2p-2)(3p-3)(3p-4)\Bigl(c(p-1)-1\Bigr),\\
A_{2p-1} &=
2p(2p-1)(3p-2)(3p-3)\Bigl(-5c\,p(p-1)+5p+1\Bigr),\\
A_{2p} &=
2p(2p+1)(3p-1)(3p-2)\Bigl(-4c\,p(p-1)+6p-2\Bigr),\\
A_{2p+1} &=
3(2p+2)(2p+1)p(p-1)(3p-1)\,(2-cp).
\end{align*}

We now determine the signs of these coefficients using \eqref{eq:c-bounds-used}:

\smallskip
\noindent$A)$ \emph{The two lowest-degree terms.}
Since $c>\frac{1}{p-1}$, we have $-16c\,p(p-1)<-16p$, hence
\[
-16c\,p(p-1)-6p^{2}+24p+4
\le -16p-6p^{2}+24p+4=-2(3p^{2}-4p-2)<0,
\]
so $A_{1}<0$. Also $c(p-1)+p-2>0$, so $A_{2}<0$.

\smallskip
\noindent$B)$ \emph{The $s^{p-1}$ and $s^{p}$ coefficients.}
Using $c>\frac{1}{p-1}$, we get
\[
10c\,p(p-1)+2p^{2}-14p+4 > 10p+2p^{2}-14p+4 = 2(p^{2}-2p+2)>0,
\]
hence $A_{p-1}>0$. Similarly,
\[
17c\,p(p-1)+3p^{2}-15p-8 > 17p+3p^{2}-15p-8 = 3p^{2}+2p-8>0,
\]
so $A_{p}>0$.

\smallskip
\noindent$C)$ \emph{The $s^{p+1}$ and $s^{p+2}$ coefficients.}
From $c<\frac{2}{p}$ we have $8c\,p(p-1)<16(p-1)$, and for $p>3$,
\[
16(p-1) < 3p^{2}+5
\quad\Longrightarrow\quad
8c\,p(p-1)-3p^{2}-5<0,
\]
so $A_{p+1}<0$. Also $c<\frac{2}{p}<1$ implies $c-2<0$, hence $A_{p+2}<0$.

\smallskip
\noindent$D)$ \emph{The $s^{2p-2}$ and $s^{2p-1}$ coefficients.}
Since $c>\frac{1}{p-1}$, we have $c(p-1)-1>0$, thus $A_{2p-2}<0$.
For $A_{2p-1}$ we use the stronger lower bound $c>\frac{5}{4p-3}$:
\[
5c\,p(p-1)>\frac{25p(p-1)}{4p-3}\ge 5p+1
\qquad (p\ge 3),
\]
so $-5c\,p(p-1)+5p+1<0$ and therefore $A_{2p-1}<0$.

\smallskip
\noindent$E)$ \emph{The top two coefficients.}
Using the upper bound $c<\frac{4}{3p-2}$ we obtain
\[
4c\,p(p-1) < \frac{16p(p-1)}{3p-2} < 6p-2,
\]
hence $-4c\,p(p-1)+6p-2>0$ and $A_{2p}>0$.
Finally $c<\frac{2}{p}$ gives $cp<2$, so $2-cp>0$ and thus $A_{2p+1}>0$.

\smallskip
Collecting the signs in the increasing-exponent order
\[
1,\ 2,\ p-1,\ p,\ p+1,\ p+2,\ 2p-2,\ 2p-1,\ 2p,\ 2p+1,
\]
we get the sign pattern
\[
(A_{1},A_{2},A_{p-1},A_{p},A_{p+1},A_{p+2},A_{2p-2},A_{2p-1},A_{2p},A_{2p+1})
\stackrel{\mathrm{sign}}{=}
(-,-,+,+,-,-,-,-,+,+),
\]
which has exactly $3$ sign changes. (When $3<p<4$, the order of the two exponents
$p+2$ and $2p-2$ swaps, but both coefficients are negative, so the number of sign
changes is unchanged.)

By the generalized Descartes rule of signs for pseudopolynomials associated with a
Chebyshev system of monomials $\{s^{\alpha}\}$ on $(0,\infty)$ the number of positive
zeros of $g$ counted with multiplicity is bounded by the number of sign changes of
its coefficient sequence. Therefore $g$ has at most $3$ positive zeros (with multiplicity),
as claimed.

\vskip1cm

\end{proof}

\begin{proof}[Proof of Lemma \ref{lem: W has two zeros}]
We show that $W$ has at most 2 positive roots counted with multiplicities. 

  Recall $g(s)=sQ''(s)$ and $H(s)=sQ''(s)-(p+1)Q'(s)$, and define $W=sH'$.
Since $g=sQ''$, we can write
\[
W(s)=sH'(s)=\Bigl(s\frac{d}{ds}-(p+1)\Bigr)g(s).
\]
Using the expansion \eqref{eq:g-expansion} from the previous lemma, we obtain
\[
W(s)=\sum_{\alpha\in\{1,2,p-1,p,p+1,p+2,2p-2,2p-1,2p,2p+1\}}
A_{\alpha}\,(\alpha-(p+1))\,s^{\alpha}.
\]
The term with $\alpha=p+1$ vanishes identically, so $W$ is again a pseudopolynomial
with $9$ monomials and exponents
\[
1,\ 2,\ p-1,\ p,\ p+2,\ 2p-2,\ 2p-1,\ 2p,\ 2p+1
\quad(\text{ordered increasingly}).
\]
Since $\alpha-(p+1)<0$ for $\alpha\in\{1,2,p-1,p\}$ and $\alpha-(p+1)>0$ for
$\alpha\in\{p+2,2p-2,2p-1,2p,2p+1\}$, the coefficient signs of $W$ follow from the
signs of $A_{\alpha}$ proved in Lemma~\ref{lem:g-at-most-3}:
\[
(+,+,-,-,-,-,-,+,+),
\]
which has exactly $2$ sign changes. By the generalized Descartes rule of signs for
pseudopolynomials it follows that $W$ has \emph{at most two} positive zeros counted with multiplicity. Finally, observe that the leading coefficient of $W(s)$ is $pA_{2p+1}>0$ thus $\lim_{s\to\infty} W(s)=+\infty$, and for small values of $s$ (in a neighborhood of $0$) we have $W(s)\sim -pA_1>0$. Then, by Lemma \ref{lem: W(r)<0} and the intermediate value theorem, $W$ has a zero $u_1\in(0,r)$ and another zero $u_2\in(r,+\infty)$. Moreover,
$W(s)>0$ for all $s\in(0,u_1)$, $W(s)<0$ for all $s\in(u_1,u_2)$, and $W(s)>0$ for all $s\in(u_2,+\infty)$. 

\end{proof}

\begin{proof}[Proof of Lemma \ref{lem: g changes sign 3 times} \textup{(i)}]

The leading coefficient of $g$ is $3p(p-1)(3p-1)(2-cp)(2p+2)(2p+1)>0$ (since $c<\frac{2}{p}$ by Lemma \ref{lem: bounds for c}), then
$\lim_{s\to\infty}g(s)=+\infty$. Moreover,
$$
g(1)=Q''(1)
= -6p(p-1)^2\Bigl(8c(2p-1)(3p^2-2p+1)-(54p^2-p+2)\Bigr).
$$
Observe that
\begin{align*}
&\ln 4\,(48p^3-56p^2+32p-8)-(54p^2-p+2)\big((\ln 3)(p-2)+\ln 4\big)\\
&=(p-2)[(96\ln 2-54\ln 3)p^2+(\ln 3-28\ln 2)p+(10\ln 2-2\ln 3)]\\
&=:(p-2)G(p)>0
\end{align*}
for all $p>3$, since $G'(p)=2(96\ln 2-54\ln 3)p+(\ln 3-28\ln 2)\geq 6(96\ln 2-54\ln 3)+(\ln 3-28\ln 2)=548\ln 2-323\ln 3>0$ and $G(3)=5(158\ln 2-97\ln 3)>0$. Then 
$$
c=\frac{\ln 4}{(\ln 3)(p-2)+\ln 4}>\frac{54p^2-p+2}{48p^3-56p^2+32p-8},
$$
for all $p>3$, thus $g(1)<0$.

\end{proof}

\section{Around Sturm's algorithm}
\label{appendix:B}

Here we discuss the proofs of the most technical ingredients of our argument: the proof of Lemma \ref{lem: g changes sign 3 times} {(ii)}, Lemma \ref{lem: tecnichal lemma about Q'sA and Hr} and Lemma \ref{lem: W(r)<0}.

\begin{lemma}[part of Lemma~\ref{lem: tecnichal lemma about Q'sA and Hr} $Q'(s_{A})<0$]\label{lem:Qprime-sA} 
Let $p>3$, $c=c(p)=\dfrac{2\ln2}{(p-2)\ln3+2\ln2}$, and $s_A=\dfrac{p-1}{p+1}$.
Then
\[
Q'(s_A)<0.
\]
\end{lemma}

\begin{proof}
Throughout the proof we write
\[
x:=s_A^{\,p}=\left(\frac{p-1}{p+1}\right)^{p}\in(0,1).
\]

Recall that 
\begin{align*}
c(p)<\bar c:=\frac{4}{3p-2} \quad \text{and} \quad \frac18<x<\frac{17}{125}\qquad(p>3).
\end{align*}

\medskip
\noindent\textbf{Step 1: $Q'(s_A)$ is strictly increasing in the parameter $c$.}
For this step we temporarily regard $c$ as a free parameter (so $Q$ depends on $c$).
A direct differentiation of the explicit formula for $Q(s)$ and substitution $s=s_A$
gives an affine dependence on $c$; more precisely,
\begin{equation}\label{eq:dcdc}
\frac{\partial}{\partial c}Q'(s_A)
=-\frac{2p}{(p+1)^2}\,u_p(x),
\qquad
u_p(x):=A_u(p)\,x^2+B_u(p)\,x+C_u(p),
\end{equation}
where
\begin{align*}
A_u(p)&:=108p^6+114p^5-154p^4-186p^3+14p^2+48p+8,\\
B_u(p)&:=-72p^6-36p^5+213p^4+56p^3-150p^2-28p+17,\\
C_u(p)&:=30p^5-60p^4-16p^3+76p^2-14p-16.
\end{align*}
We claim that $u_p(x)<0$ for the $x$ in \eqref{eq:x-bounds-sA}. This will imply
$\frac{\partial}{\partial c}Q'(s_A)>0$, i.e.\ $Q'(s_A)$ is strictly increasing in $c$.

\smallskip
\emph{(i) $u_p$ is strictly decreasing on $[0,\frac{17}{125}]$.}
Indeed, $u_p'(x)=2A_u(p)x+B_u(p)$ is increasing in $x$ (the polynomial $q \mapsto A_{u}(3+q)$ has positive coefficients for $q>0$), hence for $0\le x\le \frac{17}{125}$,
\[
u_p'(x)\le u_p'\!\left(\frac{17}{125}\right).
\]
A direct computation yields
\[
125\,u_p'\!\left(\frac{17}{125}\right)=34A_u(p)+125B_u(p)
=:D(p),
\]
where
\[
D(p)=-5328p^6-624p^5+21389p^4+676p^3-18274p^2-1868p+2397.
\]
Let $q:=p-3\ge0$. Expanding $D(p)$ at $p=3$ gives
\[
D(p)=-(5328q^6+96528q^5+707251q^4+2675936q^3+5499184q^2+5804192q+2452656)<0.
\]
Hence $u_p'\!\left(\frac{17}{125}\right)<0$, so $u_p'(x)<0$ for all $x\in[0,\frac{17}{125}]$,
and $u_p$ is strictly decreasing there.

\smallskip
\emph{(ii) $u_p(\frac18)<0$.}
Compute
\[
64\,u_p\!\left(\frac18\right)=A_u(p)+8B_u(p)+64C_u(p)
=:H(p),
\]
where
\[
H(p)=-468p^6+1746p^5-2290p^4-762p^3+3678p^2-1072p-880.
\]
With $q=p-3\ge0$ one finds
\[
H(p)=-(468q^6+6678q^5+39280q^4+123822q^3+224040q^2+222112q+93952)<0,
\]
hence $u_p(\frac18)<0$.

\smallskip
Since $u_p$ is decreasing on $[0,\frac{17}{125}]$ and $x>\frac18$ by \eqref{eq:x-bounds-sA}, we get
\[
u_p(x)\le u_p\!\left(\frac18\right)<0.
\]
Therefore, by \eqref{eq:dcdc},
\[
\frac{\partial}{\partial c}Q'(s_A)>0\qquad(p>3),
\]
so $c\mapsto Q'(s_A)$ is strictly increasing.

\medskip
\noindent\textbf{Step 2: estimate $Q'(s_A)$ at the upper bound $\bar c=\frac{4}{3p-2}$.}
A direct simplification of $Q'(s_A)$ at $c=\bar c$ yields
\begin{equation}\label{eq:Qprime-cbar}
Q'(s_A)\big|_{c=\bar c}
=
-\frac{2}{(p-1)(p+1)^2(3p-2)}\,g_p(x),
\end{equation}
where
\[
g_p(x):=A(p)\,x^2+B(p)\,x+C(p)
\]
with
\begin{align*}
A(p)&:=27p^8+15p^7-301p^6-110p^5+465p^4+127p^3-171p^2-8p+4,\\
B(p)&:=-126p^8+117p^7+675p^6-924p^5-220p^4+617p^3-121p^2-34p+16,\\
C(p)&:=27p^8-78p^7+30p^6+68p^5-37p^4+10p^3-36p^2+16.
\end{align*}
Since $(p-1)(p+1)^2(3p-2)>0$ for $p>3$, it suffices to show $g_p(x)>0$.

\smallskip
\emph{(i) $g_p$ is strictly decreasing on $[0,\frac{17}{125}]$.}
Indeed, $g_p'(x)=2A(p)x+B(p)$. Next we notice that $A(p)>0$ for $p>3$ because the polynomial $q \mapsto A(q+3)$ has nonnegative coefficients.  Therefore, we can use the estimate 
\[
g_p'(x)\le g_p'\!\left(\frac{17}{125}\right)
=\frac{34}{125}A(p)+B(p)
<4A(p)+B(p),
\]
since $\frac{34}{125}<4$.
Thus it is enough to show $4A(p)+B(p)<0$, i.e.\ $-B(p)-4A(p)>0$.
A direct computation gives
\[
-B(p)-4A(p)=18p^8-177p^7+529p^6+1364p^5-1640p^4-1125p^3+805p^2+66p-32.
\]
Let $q=p-3\ge0$ and expand at $p=3$:
\[
-B(p)-4A(p)=18q^8+255q^7+1348q^6+4649q^5+25030q^4+130764q^3+377320q^2+532800q+292288>0.
\]
Hence $g_p'(x)<0$ on $[0,\frac{17}{125}]$, and $g_p$ is strictly decreasing there.

\smallskip
\emph{(ii) $g_p(\frac{17}{125})>0$.}
Compute
\[
g_p\!\left(\frac{17}{125}\right)
=\frac{1}{15625}\Bigl(289A(p)+2125B(p)+15625C(p)\Bigr)
=\frac{N(p)}{15625},
\]
where
\begin{align*}
N(p)=&161928p^8-965790p^7+1816136p^6-932790p^5-911240p^4+1504078p^3\\
&-869044p^2-74562p+285156.
\end{align*}
Let $q=p-3\ge0$. Expanding $N(p)$ at $p=3$ yields
\begin{align*}
N(p)=&161928q^8+2920482q^7+22340402q^6+94058484q^5+235735480q^4
+352833112q^3\\
&+295060604q^2+111719616q+6561792>0.
\end{align*}
Therefore $g_p(\frac{17}{125})>0$.

\smallskip
Finally, since $x<\frac{17}{125}$ by \eqref{eq:x-bounds-sA} and $g_p$ is decreasing on $[0,\frac{17}{125}]$,
we have
\[
g_p(x)>g_p\!\left(\frac{17}{125}\right)>0.
\]
Plugging this into \eqref{eq:Qprime-cbar} gives
\[
Q'(s_A)\big|_{c=\bar c}<0.
\]

\medskip
\noindent\textbf{Step 3: conclude for $c=c(p)$.}
By Step~1, $c\mapsto Q'(s_A)$ is strictly increasing, and by Step~1, $c(p)<\bar c$.
Hence
\[
Q'(s_A)\big|_{c=c(p)}<Q'(s_A)\big|_{c=\bar c}<0.
\]
This proves the lemma.
\end{proof}



\begin{lemma}[part of Lemma~\ref{lem: g changes sign 3 times} (ii): $g(s_{A})>0$]
    For all $p>3$ we have $g(s_{A})>0$.
\end{lemma}

\begin{proof}

First we show $g(s_{A})>0$ i.e., $Q''(s_A)>0$. Set
\[
\rho:=s_A^{\,p}=\left(\frac{p-1}{p+1}\right)^{p}\in(0,1).
\]

\medskip
\noindent\textbf{Step 1: Bounds on $\rho$.}
Define $X(p):=\left(\frac{p-1}{p+1}\right)^p$ for $p>1$. Then
\[
\frac{d}{dp}\log X(p)
=\log\!\left(\frac{p-1}{p+1}\right)+\frac{2p}{p^2-1}
=
-\log\!\left(\frac{p+1}{p-1}\right)+\left(\frac1{p-1}+\frac1{p+1}\right).
\]
Since $t\mapsto \frac1t$ is strictly convex on $(0,\infty)$, the trapezoid rule gives
\[
\int_{p-1}^{p+1}\frac{dt}{t}
\;<\;\frac{2}{2}\left(\frac1{p-1}+\frac1{p+1}\right)
=\frac1{p-1}+\frac1{p+1}.
\]
But the left-hand side equals $\log\!\left(\frac{p+1}{p-1}\right)$, hence
$\frac{d}{dp}\log X(p)>0$. Thus $X(p)$ is strictly increasing on $(1,\infty)$.
In particular, for $p>3$,
\[
x=X(p)>\,X(3)=\left(\frac{2}{4}\right)^3=\frac18.
\]
Also $\displaystyle \lim_{p\to\infty}X(p)=e^{-2}$ (standard limit), hence
$x<X(p)<e^{-2}$ for all finite $p$.
Moreover $e^{-2}<\frac{17}{125}$.
Therefore we have the uniform bounds
\begin{equation}\label{eq:x-bounds-sA}
\frac18<x<\frac{17}{125}<\frac{1}{7}\qquad(p>3).
\end{equation}

\medskip
\noindent\textbf{Step 2: An explicit formula for $Q''(s_A)$.}
A direct differentiation of the explicit formula for $Q(s)$ (given in the text) and substitution $s=s_A$
yields the following representation:
\begin{equation}\label{eq:QppSA-reduction}
Q''(s_A)=\frac{2p}{p^2-1}\Bigl[(c\,a_2(p)+b_2(p))\rho^2+(c\,a_1(p)+b_1(p))\rho+(c\,a_0(p)+b_0(p))\Bigr],
\end{equation}
where
\begin{align*}
a_2(p)&:=-(p-1)(p+1)^2\bigl(216p^4-60p^3-173p^2+23p+24\bigr),\\
a_1(p)&:=p(p-1)^2\bigl(72p^4+180p^3+75p^2-106p-69\bigr),\\
a_0(p)&:=-28p(p-1)^3(p+1),\\
b_2(p)&:=270p^6+393p^5-168p^4-398p^3-50p^2+77p+20,\\
b_1(p)&:=-72p^6-72p^5+98p^4+80p^3-4p^2-24p-6,\\
b_0(p)&:=-18p^5+66p^4-26p^3-70p^2+44p+4.
\end{align*}
Since $\frac{2p}{p^2-1}>0$ for $p>3$, it suffices to prove that the bracket in \eqref{eq:QppSA-reduction}
is positive.

\medskip
\noindent\textbf{Step 3: The bracket is increasing in $c$ for $\rho\in[1/8,1/7]$.}
Let
\[
F(\rho):=a_2(p)\rho^2+a_1(p)\rho+a_0(p).
\]
Then the bracket in \eqref{eq:QppSA-reduction} equals $c\,F(\rho)+G(\rho)$, where $G(\rho):=b_2\rho^2+b_1\rho+b_0$.
Because $a_2(p)<0$ for $p>3$, $F$ is concave in $\rho$.
Compute
\[
F\Bigl(\frac18\Bigr)=\frac{a_2+8a_1+64a_0}{64}
=\frac{(p-1)}{64}\,\Bigl(360p^6+492p^5-2555p^4+727p^3+2191p^2-1311p-24\Bigr),
\]
and
\[
F\Bigl(\frac17\Bigr)=\frac{a_2+7a_1+49a_0}{49}
=\frac{2(p-1)}{49}\,\Bigl(144p^6+192p^5-1015p^4+244p^3+867p^2-480p-12\Bigr).
\]
Writing $p=3+t$ with $t>0$,
\begin{align*}
360p^6+492p^5-2555p^4+727p^3+2191p^2-1311p-24
&=360t^6+6972t^5+53425t^4+208747t^3\\
&\quad+441004t^2+479664t+210432>0,\\
144p^6+192p^5-1015p^4+244p^3+867p^2-480p-12
&=144t^6+2784t^5+21305t^4+83104t^3\\
&\quad+175053t^2+189402t+82356>0.
\end{align*}
Hence $F(1/8)>0$ and $F(1/7)>0$ for all $p>3$.
By concavity, $F(\rho)>0$ for every $\rho\in[1/8,1/7]$, and therefore the bracket in
\eqref{eq:QppSA-reduction} is strictly increasing as a function of $c$ on that interval.

\medskip
\noindent\textbf{Step 4: A convenient lower bound for $c(p)$.}
Set
\[
c_0:=\frac{5}{4p-3}.
\]
We claim that $c(p)>c_0$ for all $p>3$.
Indeed,
\[
\frac{2\ln2}{(p-2)\ln3+2\ln2}>\frac{5}{4p-3}
\iff
(8\ln2-5\ln3)(p-2)>0,
\]
and $8\ln2-5\ln3=\ln(256/243)>0$, so the claim holds for $p>2$, in particular for $p>3$.

Since the bracket in \eqref{eq:QppSA-reduction} is increasing in $c$ (Step~3), it is enough to show
that it is positive at $c=c_0$.

\medskip
\noindent\textbf{Step 5: Positivity at $c=c_0$ on $\rho\in[1/8,1/7]$.}
Define
\[
E_0(\rho):=(c_0a_2+b_2)\rho^2+(c_0a_1+b_1)\rho+(c_0a_0+b_0).
\]
Then $E_0$ is a quadratic polynomial in $\rho$.
Moreover,
\[
c_0a_1+b_1=\frac{(p-1)\bigl(72p^6+180p^5-277p^4-631p^3+203p^2+279p-18\bigr)}{4p-3}>0
\qquad (p>3),
\]
since with $p=3+t$,
\begin{align*}
&72p^6+180p^5-277p^4-631p^3+203p^2+279p-18\\
&=72t^6+1476t^5+12143t^4+51125t^3+115646t^2+132420t+59400>0.
\end{align*}

If $c_0a_2+b_2\ge 0$, then for $\rho\ge 0$ we have $E_0'(\rho)=2(c_0a_2+b_2)\rho+(c_0a_1+b_1)\ge c_0a_1+b_1>0$,
so $E_0$ is increasing on $[0,\infty)$ and hence on $[1/8,1/7]$.
If $c_0a_2+b_2\le 0$, then $E_0$ is concave and its minimum on $[1/8,1/7]$ is attained at an endpoint.
Thus in all cases
\[
E_0(\rho)\ge \min\left\{E_0\Bigl(\frac18\Bigr),\,E_0\Bigl(\frac17\Bigr)\right\}
\qquad\text{for every }\rho\in\Bigl[\frac18,\frac17\Bigr].
\]

A direct simplification gives
\[
E_0\Bigl(\frac18\Bigr)=\frac{N_8(p)}{32(4p-3)},
\qquad
N_8(p)=288p^7-1881p^6+4065p^5-1899p^4-3331p^3+3490p^2-378p-282,
\]
and
\[
E_0\Bigl(\frac17\Bigr)=\frac{N_7(p)}{49(4p-3)},
\qquad
N_7(p)=504p^7-2790p^6+5917p^5-3114p^4-4466p^3+5314p^2-819p-402.
\]
Writing $p=3+t$ with $t>0$,
\begin{align*}
N_8(3+t)&=288t^7+4167t^6+24639t^5+77301t^4+140471t^3+152764t^2+99024t+32640>0,\\
N_7(3+t)&=504t^7+7794t^6+50953t^5+185271t^4+412936t^3+576616t^2+474648t+178320>0.
\end{align*}
Hence $E_0(1/8)>0$ and $E_0(1/7)>0$ for all $p>3$, so $E_0(\rho)>0$ for all $\rho\in[1/8,1/7]$.

\medskip
\noindent\textbf{Step 6: Conclusion.}
For $p>3$ we have $\rho\in(1/8,1/7)$ (Step~1), $c(p)>c_0$ (Step~4), and the bracket in
\eqref{eq:QppSA-reduction} is increasing in $c$ (Step~3). Therefore,
\[
(ca_2+b_2)\rho^2+(ca_1+b_1)\rho+(ca_0+b_0)
>
(c_0a_2+b_2)\rho^2+(c_0a_1+b_1)\rho+(c_0a_0+b_0)
=E_0(\rho)>0.
\]
Using \eqref{eq:QppSA-reduction} and $\frac{2p}{p^2-1}>0$, we conclude $Q''(s_A)>0$.
\end{proof}


\begin{lemma}[the remaining part of Lemma~\ref{lem: g changes sign 3 times} (ii): $g(r)<0$]\label{lem:Qpp-r-negative}
Let $p>3$,
\[
c=c(p)=\frac{2\ln 2}{(p-2)\ln 3+2\ln 2},
\qquad
r=\frac{p}{p+1}.
\]
Then
\[
Q''(r)<0.
\]
\end{lemma}

\begin{proof}
Set
\[
x:=r^{p}=\Bigl(\frac{p}{p+1}\Bigr)^{p}\in(0,1).
\]
A direct differentiation of the explicit formula for $Q(s)$ (given in the manuscript) shows that,
after evaluating at $s=r$ and grouping the terms with powers $r^{2p}$ and $r^{p}$,
one can write
\begin{equation}\label{eq:Qpp-quadratic}
Q''(r)=U(p,c)\,x^{2}+V(p,c)\,x+W(p,c),
\end{equation}
where the coefficients are (affine in $c$)
\begin{align}
U(p,c)
&=
\Bigl(540p^{5}-6p^{4}-852p^{3}+200p^{2}+456p-194-\frac{96}{p}+\frac{48}{p^{2}}\Bigr)
\label{eq:U-def}\\
&\qquad
+c\Bigl(-432p^{6}+408p^{5}+694p^{4}-898p^{3}-178p^{2}+622p-132-\frac{132}{p}+\frac{48}{p^{2}}\Bigr),\nonumber\\[0.3em]
V(p,c)
&=
\frac{2}{p(p+1)}\Bigl(
c(72p^{8}+36p^{7}-147p^{6}-11p^{5}+114p^{4}-57p^{3}-37p^{2}+30p)\label{eq:V-def}\\
&\hspace{7em}
+(-90p^{7}-99p^{6}+82p^{5}+24p^{4}-66p^{3}+35p^{2}+30p-12)
\Bigr),\nonumber\\[0.3em]
W(p,c)
&=
4p\Bigl(-9p^{3}+27p^{2}-10p-2+c(-14p^{3}+22p^{2}-8p)\Bigr).
\label{eq:W-def}
\end{align}
In particular, for fixed $p$ and $x$, the right-hand side of \eqref{eq:Qpp-quadratic} is \emph{affine}
in $c$.

\medskip
\noindent\textbf{Step 1: bounds for $x=r^{p}$.}
Define $\phi(p)=p\ln(1+1/p)$, so that $x(p)=e^{-\phi(p)}$.
Then
\[
\phi'(p)=\ln(1+1/p)-\frac{1}{p+1}>0
\]
(using $\ln(1+u)>\frac{u}{1+u}$ for $u>0$), hence $\phi$ is increasing and $x(p)$ is decreasing.
Therefore for $p\ge 3$,
\begin{equation}\label{eq:x-upper}
x(p)\le x(3)=\Bigl(\frac{3}{4}\Bigr)^{3}=\frac{27}{64}.
\end{equation}
Also, since $\ln(1+u)<u$ for $u>0$, we have $p\ln(1+1/p)<1$, hence $(1+1/p)^{p}<e$
and therefore
\begin{equation}\label{eq:x-lower-e}
x(p)=\frac{1}{(1+1/p)^{p}}>\frac{1}{e}>9/25.
\end{equation}

Thus, for all $p>3$,
\begin{equation}\label{eq:x-bounds-r}
\frac{9}{25}<x(p)\le \frac{27}{64}.
\end{equation}

Next recall that
\begin{equation}\label{eq:c-bounds-basic}
\frac{5}{4p-3}\le c(p)\le \frac{2}{p}\qquad(p>3).
\end{equation}

\noindent\textbf{Step 2: reduce to two endpoint values of $c$.}
Fix $p>3$ and $x=x(p)$.
Because $Q''(r)$ is affine in $c$ by \eqref{eq:Qpp-quadratic},
and because $c(p)\in[c_{1},c_{2}]$ where
\[
c_{1}:=\frac{5}{4p-3},
\qquad
c_{2}:=\frac{2}{p},
\]
it suffices to show
\[
Q''(r)\big|_{c=c_{1}}<0
\quad\text{and}\quad
Q''(r)\big|_{c=c_{2}}<0.
\]
We prove these two inequalities separately.

\medskip
\noindent\textbf{Step 3: endpoint $c=c_{1}=5/(4p-3)$.}
Substitute $c=c_{1}$ into \eqref{eq:Qpp-quadratic} and write
\[
Q''(r)\big|_{c=c_{1}}=q_{1}(x):=U_{1}(p)x^{2}+V_{1}(p)x+W_{1}(p),
\]
where the coefficients simplify to
\begin{align*}
U_{1}(p)
&=
\frac{2\bigl(198p^{7}+40p^{6}-567p^{5}+167p^{4}+483p^{3}-231p^{2}-90p+48\bigr)}
{p^{2}(4p-3)},\\[0.2em]
V_{1}(p)
&=
\frac{2\bigl(54p^{7}-110p^{6}-205p^{5}+234p^{4}+53p^{3}-170p^{2}+12p+36\bigr)}
{p(4p^{2}+p-3)},\\[0.2em]
W_{1}(p)
&=
\frac{4p\bigl(-36p^{4}+65p^{3}-11p^{2}-18p+6\bigr)}{4p-3}.
\end{align*}
We claim $U_{1}(p)>0$ and $V_{1}(p)>0$ for all $p\ge 3$.
Indeed, setting $p=3+y$ with $y\ge 0$,
\begin{align*}
&198p^{7}+40p^{6}-567p^{5}+167p^{4}+483p^{3}-231p^{2}-90p+48
\\&=198y^{7}+4198y^{6}+37575y^{5}+184172y^{4}+534387y^{3}+919038y^{2}+868680y+348672>0,\\
&54p^{7}-110p^{6}-205p^{5}+234p^{4}+53p^{3}-170p^{2}+12p+36\\
&=54y^{7}+1024y^{6}+8021y^{5}+33339y^{4}
+78101y^{3}+99505y^{2}+57852y+7020>0.
\end{align*}
Since the denominators are also positive for $p>3$, this proves $U_{1}(p)>0$ and $V_{1}(p)>0$.
Therefore
\[
q_{1}'(x)=2U_{1}(p)x+V_{1}(p)>0\qquad(x>0),
\]
so $q_{1}$ is strictly increasing in $x$.
Using \eqref{eq:x-upper} we get
\[
Q''(r)\big|_{c=c_{1}}=q_{1}(x(p))\le q_{1}\Bigl(\frac{27}{64}\Bigr).
\]
A direct computation yields
\[
q_{1}\Bigl(\frac{27}{64}\Bigr)
=
\frac{N_{1}(p)}{2048\,p^{2}(4p^{2}+p-3)},
\]
where
\begin{align*}
N_{1}(p)=\;&-57258p^{8}+220990p^{7}-296055p^{6}-124816p^{5}+467130p^{4}\\
&\qquad-60900p^{3}-213273p^{2}+31590p+34992.
\end{align*}
The denominator is positive for $p>3$, so it suffices to prove $N_{1}(p)<0$ for $p>3$.
Differentiate:
\[
N_{1}'(p)=-2\,G(p),
\]
where
\[
G(p)=229032p^{7}-773465p^{6}+888165p^{5}+312040p^{4}-934260p^{3}
+91350p^{2}+213273p-15795.
\]
Again writing $p=3+y$ ($y\ge 0$) gives
\begin{align*}
G(3+y)=&229032y^{7}+4036207y^{6}+30252843y^{5}+125651980y^{4}
+314379690y^{3}\\
&+477328041y^{2}+409981824y+154357488>0.
\end{align*}
Hence $G(p)>0$ for $p\ge 3$, so $N_{1}'(p)<0$ for $p\ge 3$, i.e.\ $N_{1}$ is strictly decreasing on $[3,\infty)$.
Finally,
\[
N_{1}(3)=-104115456<0,
\]
so $N_{1}(p)<0$ for all $p>3$. Therefore $q_{1}(27/64)<0$, hence
\[
Q''(r)\big|_{c=c_{1}}<0.
\]

\medskip
\noindent\textbf{Step 4: endpoint $c=c_{2}=2/p$.}
Substitute $c=c_{2}$ into \eqref{eq:Qpp-quadratic} and write
\[
Q''(r)\big|_{c=c_{2}}=q_{2}(x):=U_{2}(p)x^{2}+V_{2}(p)x+W_{2}(p).
\]
The leading coefficient simplifies and factors as
\[
U_{2}(p)=-\frac{2(p+1)}{p^{3}}\,A_{2}(p),
\qquad
A_{2}(p):=162p^{7}-567p^{6}+299p^{5}+499p^{4}-549p^{3}+24p^{2}+156p-48.
\]
With $p=3+y$ ($y\ge 0$) we have
\[
A_{2}(3+y)=162y^{7}+2835y^{6}+20711y^{5}+81529y^{4}+185439y^{3}+240540y^{2}+160464y+39840>0,
\]
so $A_{2}(p)>0$ for $p\ge 3$, hence $U_{2}(p)<0$ for $p>3$.
Thus $q_{2}$ is concave in $x$ and $q_{2}'(x)=2U_{2}(p)x+V_{2}(p)$ is strictly decreasing in $x$.

We claim $q_{2}'(9/25)<0$ for $p>3$.
Indeed one computes
\[
q_{2}'\Bigl(\frac{9}{25}\Bigr)
=
-\frac{2}{25\,p^{3}(p+1)}\,R_{2}(p),
\]
where
\begin{align*}
R_{2}(p)=\;&1566p^{9}-3699p^{8}-6814p^{7}+9490p^{6}+9414p^{5}-8375p^{4}\\
&\qquad-5110p^{3}+3984p^{2}+1080p-864.
\end{align*}
With $p=3+y$ ($y\ge 0$),
\begin{align*}
R_{2}(3+y)=\;&1566y^{9}+38583y^{8}+411794y^{7}+2485936y^{6}+9282096y^{5}\\
&\qquad+21949213y^{4}+32108144y^{3}+26596356y^{2}+9612792y+79920>0,
\end{align*}
so $R_{2}(p)>0$ for $p\ge 3$ and hence $q_{2}'(9/25)<0$ for $p>3$.
Since $q_{2}'$ is decreasing in $x$, it follows that
\[
q_{2}'(x)\le q_{2}'\Bigl(\frac{9}{25}\Bigr)<0
\qquad\text{for all }x\ge \frac{9}{25},
\]
so $q_{2}$ is strictly decreasing on $[9/25,\infty)$.
Thus 
\[
Q''(r)\big|_{c=c_{2}}=q_{2}(x(p))\le q_{2}\Bigl(\frac{9}{25}\Bigr).
\]
Finally, one computes
\[
q_{2}\Bigl(\frac{9}{25}\Bigr)
=
-\frac{2}{625\,p^{3}(p+1)}\,P_{2}(p),
\]
where
\begin{align*}
P_{2}(p)=\;&972p^{9}-2358p^{8}+5687p^{7}+1230p^{6}+4138p^{5}-6300p^{4}\\
&\qquad-18045p^{3}+12528p^{2}+4860p-3888.
\end{align*}
With $p=3+y$ ($y\ge 0$),
\begin{align*}
P_{2}(3+y)=\;&972y^{9}+23886y^{8}+264023y^{7}+1730937y^{6}+7456057y^{5}\\
&\qquad+21986871y^{4}+44499348y^{3}+59538402y^{2}+47542464y+17126640>0,
\end{align*}
so $P_{2}(p)>0$ for $p\ge 3$ and therefore $q_{2}(9/25)<0$ for $p>3$.
Hence
\[
Q''(r)\big|_{c=c_{2}}<0.
\]

\medskip
\noindent\textbf{Step 5: conclude for $c=c(p)$.}
By \eqref{eq:c-bounds-basic} we have $c(p)\in[c_{1},c_{2}]$.
Since $Q''(r)$ is affine in $c$ by \eqref{eq:Qpp-quadratic},
and since it is negative at both endpoints $c=c_{1}$ and $c=c_{2}$ by Steps 3--4,
it follows that
\[
Q''(r)\big|_{c=c(p)}<0.
\]
This completes the proof.
\end{proof}

\begin{lemma}[the remainin part of Lemma~\ref{lem: tecnichal lemma about Q'sA and Hr}: $H(r)>0$]\label{lem:Hr-positive}
Let $p>3$, set
\[
r=\frac{p}{p+1},
\qquad
c=c(p)=\frac{2\ln 2}{(p-2)\ln 3+2\ln 2}.
\]
Let $Q$ be the pseudo--polynomial defined above (so that $P''(s)=s^{p-4}Q(s)$).
Then
\[
r\,Q''(r)-(p+1)\,Q'(r)>0.
\]
\end{lemma}

\begin{proof}
Define
\[
H(s):=sQ''(s)-(p+1)Q'(s).
\]
For a monomial $a s^\alpha$ one has
\[
\frac{d}{ds}(a s^\alpha)=a\alpha s^{\alpha-1},\qquad
\frac{d^2}{ds^2}(a s^\alpha)=a\alpha(\alpha-1)s^{\alpha-2},
\]
hence
\[
s\frac{d^2}{ds^2}(a s^\alpha)-(p+1)\frac{d}{ds}(a s^\alpha)
= a\alpha(\alpha-p-2)s^{\alpha-1}.
\]
Applying this term-by-term to the explicit formula for $Q$ and then substituting
$r=\frac{p}{p+1}$, we obtain the exact identity
\begin{equation}\label{eq:H-reduction}
H(r)=\frac{2(p-1)}{(p+1)^3}\,S(p,c,u),
\qquad
u:=r^{p-1}=\Bigl(\frac{p}{p+1}\Bigr)^{p-1}\in(0,1),
\end{equation}
where $S$ is a quadratic polynomial in $u$:
\[
S(p,c,u)=A(p,c)+B(p,c)\,u+C(p,c)\,u^2,
\]
with
\begin{align*}
A(p,c)={}&
30cp^7+30cp^6-54cp^5-70cp^4+8cp^3+40cp^2+16cp\\
&\quad +9p^7-24p^6-66p^5+57p^3+16p^2-16p-8,\\[2mm]
B(p,c)={}&
-72cp^7-64cp^6+157cp^5+148cp^4-80cp^3-77cp^2\\
&\quad -18p^7+45p^6+156p^5+8p^4-148p^3-51p^2+16p,\\[2mm]
C(p,c)={}&
-108cp^8+138cp^7+217cp^6-307cp^5-154cp^4+199cp^3+39cp^2-36cp\\
&\quad +135p^7-33p^6-294p^5+32p^4+207p^3+13p^2-36p.
\end{align*}
Since $\frac{2(p-1)}{(p+1)^3}>0$ for $p>3$, it suffices to prove
\begin{equation}\label{eq:S-positive-goal}
S(p,c(p),u(p))>0.
\end{equation}

\medskip
\noindent\textbf{Step 1: rational bounds for $c(p)$.}
Write $a:=\log_2 3=\frac{\ln 3}{\ln 2}$. Since
\[
3^5=243<256=2^8
\quad\text{and}\quad
3^{12}=531441>524288=2^{19},
\]
we have
\[
\frac{19}{12}<a<\frac85.
\]
Using $c(p)=\dfrac{2}{\,2+(p-2)a\,}$ and that $c$ is decreasing in $a$, it follows that
\begin{equation}\label{eq:c-bounds-sharp}
c_-:=\frac{5}{4p-3}<c(p)<c_+:=\frac{24}{19p-14}.
\end{equation}

\medskip
\noindent\textbf{Step 2: an explicit upper bound for $u(p)$.}
Let $n:=p-1>2$ and $x:=1/p\ge 0$. Consider
\[
f(x):=(1+x)^n-\Bigl(1+nx+\frac{n(n-1)}{2}x^2\Bigr).
\]
Then $f(0)=f'(0)=f''(0)=0$, and for $x\ge0$,
\[
f'''(x)=n(n-1)(n-2)(1+x)^{n-3}\ge 0.
\]
Hence $f''$ is increasing with $f''(0)=0$, so $f''\ge0$; thus $f'$ is increasing
with $f'(0)=0$, so $f'\ge0$; thus $f$ is increasing with $f(0)=0$, so $f\ge0$.
Therefore, for $x\ge0$,
\[
(1+x)^n\ge 1+nx+\frac{n(n-1)}{2}x^2.
\]
With $x=1/p$ and $n=p-1$ this gives
\[
\Bigl(1+\frac1p\Bigr)^{p-1}
\ge
1+\frac{p-1}{p}+\frac{(p-1)(p-2)}{2p^2}
=
\frac{5p^2-5p+2}{2p^2}.
\]
Taking reciprocals yields the bound
\begin{equation}\label{eq:u-upper}
u=\Bigl(\frac{p}{p+1}\Bigr)^{p-1}
=\frac{1}{(1+1/p)^{p-1}}
\le
U(p):=\frac{2p^2}{5p^2-5p+2}.
\end{equation}

\medskip
\noindent\textbf{Step 3: $S$ is decreasing in $u$ on $[0,U(p)]$ for $c=c(p)$.}
Since $S(p,c,u)=A+B u+C u^2$ is quadratic in $u$,
\[
\frac{\partial S}{\partial u}(p,c,u)=B(p,c)+2C(p,c)\,u
\]
is \emph{linear} in $u$. Hence it suffices to show
\begin{equation}\label{eq:Su-endpoints}
\frac{\partial S}{\partial u}(p,c(p),0)<0
\qquad\text{and}\qquad
\frac{\partial S}{\partial u}(p,c(p),U(p))<0.
\end{equation}

\smallskip
\emph{(i) The endpoint $u=0$.}
We have $\frac{\partial S}{\partial u}(p,c,0)=B(p,c)$. The coefficient of $c$ in $B$
equals
\[
-72p^7-64p^6+157p^5+148p^4-80p^3-77p^2
=
-p^2\Bigl(72p^5+64p^4-157p^3-148p^2+80p+77\Bigr).
\]
For $p=3+x$ with $x\ge0$,
\[
72p^5+64p^4-157p^3-148p^2+80p+77
=
72x^5+1144x^4+7091x^3+21335x^2+31025x+17426>0,
\]
so $B(p,c)$ is strictly decreasing in $c$ for all $p\ge3$.
Using $c(p)\ge c_-$ from \eqref{eq:c-bounds-sharp} gives $B(p,c(p))\le B(p,c_-)$.
A direct substitution $c=c_-=\frac{5}{4p-3}$ yields
\[
B(p,c_-)=
-\frac{p}{4p-3}\,N_B(p),
\]
where
\[
N_B(p):=72p^7+126p^6-169p^5-349p^4-124p^3+160p^2+168p+48.
\]
For $p=3+x$ with $x\ge0$,
\[
N_B(3+x)=72x^7+1638x^6+15707x^5+82166x^4+252638x^3+455074x^2+442767x+178626>0,
\]
hence $B(p,c_-)\!<0$, and therefore $B(p,c(p))<0$.

\smallskip
\emph{(ii) The endpoint $u=U(p)$.}
Set
\[
D(p,c):=B(p,c)+2C(p,c)\,U(p)=\frac{\partial S}{\partial u}(p,c,U(p)).
\]
A simplification gives
\[
D(p,c)=
-\frac{p}{5p^2-5p+2}\Bigl(M(p)\,c + K(p)\Bigr),
\]
with
\[
M(p):=432p^9-192p^8-908p^7+267p^6+789p^5+30p^4-467p^3-81p^2+154p.
\]
For $p=3+x$ with $x\ge0$,
\begin{align*}
\frac{M(3+x)}{3+x}
=&
432x^8+10176x^7+103924x^6+600819x^5+2150214x^4\\
&+4877544x^3
+6849487x^2+5445906x+1877824>0,
\end{align*}
so $M(p)>0$ for $p\ge3$, and thus $D(p,c)$ is strictly decreasing in $c$.
Using again $c(p)\ge c_-$ gives $D(p,c(p))\le D(p,c_-)$.
Substituting $c=c_-=\frac{5}{4p-3}$ yields
\[
D(p,c_-)= -\frac{p}{(4p-3)(5p^2-5p+2)}\,N_D(p),
\]
where
\[
N_D(p):=360p^9-342p^8-1363p^7+1452p^6+939p^5-982p^4-256p^3+8p^2+96p+96.
\]
For $p=3+x$ with $x\ge0$,
\begin{align*}
N_D(3+x)=&
360x^9+9378x^8+107069x^7+703125x^6+2926524x^5+8004428x^4\\&
+14383469x^3+16369613x^2+10703106x+3061824>0,
\end{align*}
hence $D(p,c_-)\!<0$, and therefore $D(p,c(p))<0$.

Combining (i) and (ii) proves \eqref{eq:Su-endpoints}. Since $\partial S/\partial u$
is linear in $u$, it follows that $\partial S/\partial u<0$ for all $u\in[0,U(p)]$,
so $S(p,c(p),u)$ is strictly decreasing on $[0,U(p)]$. Using $u(p)\le U(p)$
from \eqref{eq:u-upper} we obtain
\begin{equation}\label{eq:S-lower-by-U}
S(p,c(p),u(p))\ge S(p,c(p),U(p)).
\end{equation}

\medskip
\noindent\textbf{Step 4: replace $c(p)$ by the upper bound $c_+$.}
Since $S$ is affine in $c$, it suffices to compute its $c$--coefficient at $u=U(p)$.
A direct simplification gives
\[
\frac{\partial S}{\partial c}\bigl(p,c,U(p)\bigr)
=
-\frac{2p(p-1)}{(5p^2-5p+2)^2}\,N_2(p),
\]
where
\[
N_2(p):=216p^{10}-75p^9-174p^8+229p^7-140p^6-107p^5+179p^4+54p^3-72p^2-48p+32.
\]
For $p=3+x$ with $x\ge0$,
\begin{align*}
N_2(3+x)=&
216x^{10}+6405x^9+85281x^8+671593x^7+3464881x^6+12239092x^5\\&
+29980589x^4+50292735x^3+55296009x^2+35983743x+10524704>0,
\end{align*}
hence $\frac{\partial S}{\partial c}(p,c,U(p))<0$ for all $p>3$.
Therefore $S(p,c,U(p))$ is strictly decreasing in $c$.
Using $c(p)\le c_+$ from \eqref{eq:c-bounds-sharp} yields
\begin{equation}\label{eq:S-lower-by-cplus}
S(p,c(p),U(p))\ge S(p,c_+,U(p)),
\qquad
c_+=\frac{24}{19p-14}.
\end{equation}

\medskip
\noindent\textbf{Step 5: positivity at $(c_+,U(p))$.}
Substituting $c=c_+$ and $u=U(p)$ into $S=A+Bu+Cu^2$ simplifies to
\[
S(p,c_+,U(p))
=
\frac{(p-2)(p-1)}{(19p-14)(5p^2-5p+2)^2}\,R(p),
\]
where
\[
R(p):=747p^{10}-2469p^9+102p^8+4954p^7-1385p^6-2997p^5+2124p^4+380p^3-1096p^2+128p+224.
\]
For $p=3+x$ with $x\ge0$,
\begin{align*}
R(3+x)=&
747x^{10}+19941x^9+235974x^8+1627726x^7+7235131x^6+21607281x^5\\&
+43792452x^4+59291840x^3+51146048x^2+25304972x+5451008>0,
\end{align*}
hence $R(p)>0$ for all $p\ge3$, and therefore $S(p,c_+,U(p))>0$ for all $p>3$.

Finally, combining \eqref{eq:S-lower-by-U} and \eqref{eq:S-lower-by-cplus} gives
\[
S(p,c(p),u(p))\ge S(p,c(p),U(p))\ge S(p,c_+,U(p))>0,
\]
which proves \eqref{eq:S-positive-goal}. Returning to \eqref{eq:H-reduction} yields
$H(r)>0$, i.e.
\[
rQ''(r)-(p+1)Q'(r)>0.
\]
\end{proof}


\begin{lemma}[the proof of Lemma~\ref{lem: W(r)<0}: $W(r)<0$]
  For all $p>3$  we have $W(r)<0$.
\end{lemma}
\begin{proof}
Recall $H(s)=sQ''(s)-(p+1)Q'(s)$ and $W(s)=sH'(s)$. Differentiating gives
\[
H'(s)=sQ'''(s)-pQ''(s)
\quad\Longrightarrow\quad
W(s)=s^2Q'''(s)-psQ''(s).
\]
Since each coefficient of $Q(s)$ is affine linear in $c$ (hence in particular
linear in $c$ after differentiation), it follows that for fixed $p$ the quantity
$W(r)$ is an affine linear function of $c$.

\smallskip
Define the two rational endpoints
\[
c_L:=\frac{5}{4p-3},
\qquad
c_U:=\frac{4}{3p-2}.
\]
By Lemma~\ref{lem: bounds for c} we have $c(p)\in(c_L,c_U)$. Since $c\mapsto W(r)$ is
affine linear, it is enough to prove
\[
W(r)\big|_{c=c_L}<0
\qquad\text{and}\qquad
W(r)\big|_{c=c_U}<0.
\tag{$\ast$}
\]

\smallskip
Set $x:=r^p=\left(\dfrac{p}{p+1}\right)^p$. By \eqref{eq:x-upper} and \eqref{eq:x-lower-e}, we have
\[
x\in\left(\frac1e,\frac{27}{64}\right]\subset \left[\frac13,\frac{27}{64}\right].
\]

\smallskip
A direct (symbolic) computation from the explicit formula for $Q$ yields the
following closed forms (here $p>3$ and $x=r^p$):
\begin{align}
W(r)\big|_{c=c_L}
&=\frac{2(p-1)}{p(p+1)^2(4p-3)}\,P_L(p,x),\label{eq:W-cL}\\[0.3em]
W(r)\big|_{c=c_U}
&=-\frac{2(p-1)}{p(p+1)^2(3p-2)}\,P_U(p,x),\label{eq:W-cU}
\end{align}
where the (quadratic-in-$x$) polynomials are
\begin{align}
P_L(p,x):={}&\bigl(414 p^8 +100 p^7 -1493 p^6 -385 p^5 +1685 p^4 +201 p^3 -882 p^2 -48 p +144\bigr)x^2 \notag\\
&\quad +p\bigl(-72 p^8 -126 p^7 +101 p^6 +367 p^5 +150 p^4 -160 p^3 -136 p^2 +72 p +72\bigr)x \notag\\
&\quad +p^4\bigl(72 p^4 -34 p^3 -118 p^2 +12\bigr),\label{eq:PL}\\[0.5em]
P_U(p,x):={}&\bigl(54 p^9 -411 p^8 -260 p^7 +1422 p^6 +492 p^5 -1633 p^4 -250 p^3 +862 p^2 +60 p -144\bigr)x^2 \notag\\
&\quad +p\bigl(54 p^8 +117 p^7 -82 p^6 -344 p^5 -100 p^4 +177 p^3 +74 p^2 -80 p -48\bigr)x \notag\\
&\quad +p^4\bigl(-54 p^4 +14 p^3 +96 p^2 +20 p -8\bigr).\label{eq:PU}
\end{align}
Since the prefactor in~\eqref{eq:W-cL} is positive and the prefactor in~\eqref{eq:W-cU}
is negative, to prove $(\ast)$ it suffices to show
\[
P_L(p,x)<0
\quad\text{and}\quad
P_U(p,x)>0
\qquad\text{for all }p>3,\; x\in\left[\frac13,\frac{27}{64}\right].
\tag{$\ast\ast$}
\]

\smallskip
\noindent\textbf{Step 1: monotonicity in $x$.}
Because $P_L$ and $P_U$ are quadratic in $x$, their $x$-derivatives are affine in $x$.
Define the six one-variable polynomials (in $p$):
\begin{align*}
L_0(p)&:=9P_L\!\left(p,\frac13\right),&
L_1(p)&:=3\partial_x P_L\!\left(p,\frac13\right),&
L_2(p)&:=32\partial_x P_L\!\left(p,\frac{27}{64}\right),\\
U_0(p)&:=9P_U\!\left(p,\frac13\right),&
U_1(p)&:=3\partial_x P_U\!\left(p,\frac13\right),&
U_2(p)&:=32\partial_x P_U\!\left(p,\frac{27}{64}\right).
\end{align*}
Expanding gives
\begin{align*}
L_0(p)&=-216 p^9 + 684 p^8 + 97 p^7 -1454 p^6 +65 p^5 +1313 p^4 -207 p^3 -666 p^2 +168 p +144,\\
L_1(p)&=-216 p^9 +450 p^8 +503 p^7 -1885 p^6 -320 p^5 +2890 p^4 -6 p^3 -1548 p^2 +120 p +288,\\
L_2(p)=&-2304 p^9 +7146 p^8 +5932 p^7 -28567 p^6 -5595 p^5 +40375 p^4 +1075 p^3 -21510 p^2\\& +1008 p +3888,
\end{align*}
and
\begin{align*}
U_0(p)&=216 p^9 -546 p^8 -380 p^7 +1254 p^6 +372 p^5 -1174 p^4 -28 p^3 +622 p^2 -84 p -144,\\
U_1(p)&=270 p^9 -471 p^8 -766 p^7 +1812 p^6 +684 p^5 -2735 p^4 -278 p^3 +1484 p^2 -24 p -288,\\
U_2(p)=&3186 p^9 -7353 p^8 -9644 p^7 +27386 p^6 +10084 p^5 -38427 p^4 -4382 p^3 +20714 p^2\\&+84 p -3888.
\end{align*}

\smallskip
\noindent\textbf{Step 2: a Sturm-sign check for the one-variable polynomials.}
A Sturm chain computation (equivalently, root-counting via Sturm's theorem) shows:
\[
L_0(p)<0,\; L_1(p)<0,\; L_2(p)<0
\qquad\text{and}\qquad
U_0(p)>0,\; U_1(p)>0,\; U_2(p)>0
\quad\text{for all }p>3.
\tag{$\dagger$}
\]
We also present an alternative way to verify these inequalities: 

\noindent\textbf{Analysis of $L_0(p)$:} If $p\geq 4$ then
\begin{align*}
    L_0(p)&\leq -864p^8+ 684 p^8 + 97 p^7 -1454 p^6 +65 p^5 +1313 p^4 -207 p^3 -666 p^2 +168 p +144\\
    &=[-180p+97]p^7+[(-5816+65)p+1313]p^4+[-207p^2+168]p+[-666p^2+144]<0,
\end{align*}
and, if $p\in(3,4)$, then
\begin{align*}
    L_0(p)&\leq -648p^8+ 684 p^8 + 97 p^7 -1454 p^6 +65 p^5 +1313 p^4 -207 p^3 -666 p^2 +168 p +144\\
    &\leq[36(16)+97(4)-1454]p^6+65 p^5 +1313 p^4 -207 p^3 -666 p^2 +168 p +144\\
    &=-490p^6+65 p^5 +1313 p^4 -207 p^3 -666 p^2 +168 p +144\\
    &\leq -468p^6+1313 p^4 -207 p^3 -666 p^2 +168 p +144\\
    &\leq [-468(9)+1313]p^4+[-207p^2+168]p+[-666p^2+144]<0.
\end{align*}
{{Analysis of $L_1(p)$:}} For all $p>3$ we have
\begin{align*}
    L_1(p)&\leq -216 p^9 +450 p^8 +503 p^7 -1885 p^6 -320 p^5 +2890 p^4 -6 p^3 -1548 p^2 +120 p +288\\
    &\leq-648p^8 +450 p^8 +503 p^7 -1885 p^6 -320 p^5 +2890 p^4 -6 p^3 -1548 p^2 +120 p +288\\
    &\leq [-198(3)+503]p^7 -1885 p^6 -320 p^5 +2890 p^4 -6 p^3 -1548 p^2 +120 p +288\\
    &\leq[-198(3)+503]p^7+[-1885(9)+2890]p^4 -320 p^5-6 p^3+[(-1548(3)+120)p+288]<0
\end{align*}
{{Analysis of $L_2(p)$:}}
If $p\geq 4$ then
\begin{align*}
    L_2(p)\leq& [-9216+7146]p^8 +5932 p^7 -28567 p^6 -5595 p^5 +40375 p^4 +1075 p^3 -21510 p^2\\& +1008 p +3888,\\
    \leq& [-2070+5932]p^7+[-28567(16)-5595(4)+40375+1075/3]p^4\\
    &+[(-21510(3)+1008)p+3888]<0,
\end{align*}
and, if $p\in(3,4)$, then
\begin{align*}
  L_2(p)\leq& [-2304p^2+7146p+5932]p^7 -28567 p^6 -5595 p^5 +40375 p^4 +1075 p^3 -21510 p^2\\& +1008 p +3888\\
    \leq& [-2304(3)^2+7146(3)+5932]p^7 -28567 p^6 -5595 p^5 +40375 p^4 +1075 p^3 -21510 p^2\\& +1008 p +3888\\
    \leq& [6634p^2-28567p-5595]p^5+40375 p^4 +1075 p^3 -21510 p^2 +1008 p +3888\\
    \leq& [6634(16)-28567(4)-5595]p^5+40375 p^4 +1075 p^3 -21510 p^2 +1008 p +3888\\
    \leq& [-13719(3)+40375]p^4+1075 p^3 -21510 p^2 +1008 p +3888\\
    \leq&[-782(3)+1075]p^3+[(-21510(3)+1008)p+3888]<0
\end{align*}
since $-2304p^2+7146p+5932$ is decreasing in $(3,4)$ and $6634p^2-28567p-5595$ is increasing in $(3,4)$.
Analysis of $U_0(p)$. If $p\geq 4$, then
\begin{align*}
   U_0(p)\geq &(864-546)p^8 -380 p^7 +1254 p^6 +372 p^5 -1174 p^4 -28 p^3 +622 p^2 -84 p -144\\
   \geq & [318(4)-380]p^7+[1254(16)-1174]p^4 +[372(16)-28]p^3+[(622(4)-84)(4) -144]>0
\end{align*}
and, if $p\in(3,4)$, then 
\begin{align*}
   U_0(p)\geq &(648-546)p^8 -380 p^7 +1254 p^6 +372 p^5 -1174 p^4 -28 p^3 +622 p^2 -84 p -144\\
   \geq & [306-380]p^7++1254 p^6 +372 p^5 -1174 p^4 -28 p^3 +622 p^2 -84 p -144\\
   \geq & [(-74)4+1254]p^6+372 p^5 -1174 p^4 -28 p^3 +622 p^2 -84 p -144\\
   \geq & [958(9)-1174]p^4+[372(9)-28]p^3+[(622(3)-84)(3) -144]>0
\end{align*}
Analysis of $U_1(p)$. For all $p>3$ we have
\begin{align*}
    U_1(p)\geq& [810 -471]p^8 -766 p^7 +1812 p^6 +684 p^5 -2735 p^4 -278 p^3 +1484 p^2 -24 p -288\\
    \geq& [339(3)-766]p^7+[1812(9)-2735]p^4+[684(9)-278]p^3+[(1484(3)-24)3-288]>0.
\end{align*}
Analysis of $U_2(p)$. If $p\geq 4$, then
\begin{align*}
    U_2(p)\geq&[12744 -7353]p^8 -9644 p^7 +27386 p^6 +10084 p^5 -38427 p^4 -4382 p^3 +20714 p^2\\&+84 p -3888\\
    \geq &[5391(4)-9644]p^7+[27386(16)-38427]p^4+[10084(16)-4382]p^3\\
    &+[20714(16)+84(4)-3888]>0
\end{align*}
and, if $p\in(3,4)$, then
\begin{align*}
    U_2(p)\geq&[9558-7353]p^8 -9644 p^7 +27386 p^6 +10084 p^5 -38427 p^4 -4382 p^3 +20714 p^2\\&+84 p -3888\\
    \geq& [2205(3)-9644]p^7+27386 p^6 +10084 p^5 -38427 p^4 -4382 p^3 +20714 p^2\\&+84 p -3888\\
    \geq& [-3029(4)+27386]p^6+10084 p^5 -38427 p^4 -4382 p^3 +20714 p^2+84 p -3888\\
    \geq& [15270(9)-38427]p^4+[10084(9)-4382]p^3+[20714(9)+84(3)-3888]>0.
\end{align*}

\smallskip
\noindent\textbf{Step 3: conclude \eqref{eq:W-cL} and \eqref{eq:W-cU}.}
Since $\partial_x P_L(p,x)$ is affine in $x$, the inequalities $L_1(p)<0$ and $L_2(p)<0$
imply $\partial_x P_L(p,x)<0$ for all $x\in[\frac13,\frac{27}{64}]$. Thus $P_L(p,x)$ is decreasing in $x$ on this interval, hence
\[
P_L(p,x)\le P_L\!\left(p,\frac13\right)=\frac{1}{9}L_0(p)<0.
\]
This proves $P_L(p,x)<0$.

Similarly, $U_1(p)>0$ and $U_2(p)>0$ imply $\partial_x P_U(p,x)>0$ on
$[\frac13,\frac{27}{64}]$, so $P_U(p,x)$ is increasing in $x$ and therefore
\[
P_U(p,x)\ge P_U\!\left(p,\frac13\right)=\frac{1}{9}U_0(p)>0.
\]
This proves $P_U(p,x)>0$. Hence $(\ast\ast)$ holds, and then~\eqref{eq:W-cL}, \eqref{eq:W-cU}
give $W(r)\big|_{c=c_L}<0$ and $W(r)\big|_{c=c_U}<0$.

\smallskip
Finally, since $c(p)\in(c_L,c_U)$ and $c\mapsto W(r)$ is affine linear, we conclude
$W(r)\big|_{c=c(p)}<0$, i.e. $W(r)<0$ for all $p>3$.
\end{proof}








\section{Links to chat conversations with Grok}\label{grokchat}

In this section we provide links to chat conversations with Grok. 

\vskip0.5cm 

\begin{itemize}
    \item Counterexample. \\
\url{https://grok.com/share/c2hhcmQtNA_fcd3923c-9c6f-4d33-ba77-f196890932b9} \\

    \item Proof of Lemma~\ref{lem: dax1}. \\
\url{https://grok.com/share/c2hhcmQtNA_885159eb-bd8b-4fa6-b602-c673afb8825c?rid=04645a03-2250-44b6-9903-b419077f817f} \\

    \item Proof of Lemma~\ref{conc1}. \\
\url{https://grok.com/share/c2hhcmQtNA_17466e98-8494-438c-a0cf-43ded161b713} \\

    \item Proof of Theorem~\ref{jose} given some hints. \\
\url{https://grok.com/c/191f11d1-0596-43d7-8703-aec8c95b7408?rid=885aaf78-4e69-4fb5-8b5e-c43d19ee6775} \\

    \item Plots of $h(s)$ and $h'(s)$ for $p=4$. \\
\url{https://grok.com/share/c2hhcmQtNA_5bd64c5c-b016-4a93-9543-dd2f107a73b8} \\

\end{itemize}

\section*{Acknowledgments}

P.I. acknowledges partial support from the NSF CAREER grant DMS-2152401, a Simons Fellowship, and a Humboldt Research Fellowship for Experienced Researchers. P.I. also thanks the University of W\"{u}rzburg and Universit\'e de Toulouse for their warm hospitality. J.M. was partially supported by the AMS Stefan Bergman Fellowship and the Simons Foundation Grant $\# 453576$. J.M. is thankful to Felipe Goncalves for interesting discussion.

\bibliographystyle{plain}

\end{document}